\documentclass[10pt,oneside]{amsart}
\usepackage[a4paper, width=5in, height=6.9in]{geometry}
\usepackage{amsmath,amssymb,amsthm,booktabs,longtable,enumerate,mathtools}
\usepackage{times}
\usepackage[loadonly]{enumitem}
\usepackage{comment}
\usepackage{graphicx}
\usepackage[all]{xy}
\usepackage{caption}
\usepackage{color}
\usepackage{hyperref}
\theoremstyle{plain}
\makeatletter 

\@addtoreset{table}{section}
\makeatother %
\newtheorem{theorem}{Theorem}[section]
\newtheorem*{theorem*}{Theorem}
\newtheorem{maintheorem}{Main Theorem}
\newtheorem*{remark*}{Remark}
\newtheorem{lemma}[theorem]{Lemma}
\newtheorem{corollary}[theorem]{Corollary}
\newtheorem{proposition}[theorem]{Proposition}
\newtheorem{definition}[theorem]{Definition}

\newtheorem{example}[theorem]{Example}

\newtheorem{problem}[theorem]{Problem}
\newtheorem{remark}[theorem]{Remark}

\newtheorem{question}[theorem]{Question}

\newcommand{\Ad}{\mathop{\mathrm{Ad}}\nolimits}
\newcommand{\rmB}{\mathop{\mathrm{B}}\nolimits}
\newcommand{\rmN}{\mathop{\mathrm{N}}\nolimits}
\newcommand{\GL}{\mathop{\mathrm{GL}}\nolimits}

\newcommand{\id}{\mathop{\mathrm{id}}\nolimits}

\newcommand{\diag}{\mathop{\mathrm{diag}}\nolimits}

\newcommand{\R}{\mathbb{R}}

\newcommand{\Z}{\mathbb{Z}}

\makeatletter
    
    \@addtoreset{equation}{section}
  \makeatother
\makeatletter
\renewcommand{\subsection}{\@startsection{subsection}{2}{0mm}%
  {\baselineskip} 
  {0.5\baselineskip} 
  {\bfseries\itshape\large}} 
\makeatother
\usepackage{enumerate}%
\makeatother
\begin{document}

\title[Proper actions on Homogeneous Spaces of $\rmB(n)$ via coarse geometry]{Proper Actions on Homogeneous Spaces of the Upper Triangular Matrix Group via Coarse Geometry}
\author{Hiroaki Nagaya}

\address[H.~Nagaya]{%
	Graduate School of Advanced Science and Engineering, Hiroshima University, 
    1-3-1 Kagamiyama, Higashi-Hiroshima City, Hiroshima, 739-8526, Japan.
        }
\email{nagayahiroaki@hiroshima-u.ac.jp}

\begin{abstract}
    We study proper actions on homogeneous spaces of the group $\rmB(n)$ of invertible upper triangular matrices from a coarse-geometric viewpoint.
    We obtain sufficient conditions for the properness of the natural $L$-action on $\rmB(n)/H$, where $L,H$ are closed subgroups of $\rmB(n)$.
    We also give an explicit sufficient condition for properness of the $L$-action on a homogeneous space naturally identified with $\rmB(n)/H\cong \mathbb{R}^{n-1}$.
    In addition, we discuss extensions to the case where the ambient group $G$ is a subgroup of $\rmB(n)$ and $L,H$ are closed subgroups of $G$, with particular attention to the unipotent subgroup $\rmN(n)$.
\end{abstract}

\newgeometry{width=5in, height=7.7in}
\maketitle

\tableofcontents

\section{Introduction}

\subsection{Historical Background}

Properly discontinuous actions play a fundamental role in the theory of discontinuous groups. Their behavior is quite different for isometric actions on Riemannian manifolds and for more general actions.
For an isometric action on a Riemannian manifold, every discrete subgroup of the isometry group acts properly discontinuously. 
However, for more general actions, a discrete group need not act properly discontinuously.
Thus, a fundamental problem in the theory of discontinuous groups is to find effective ways to construct and detect such actions. Proper actions provide a useful framework for this purpose, since the restriction of a proper action to a discrete subgroup is properly discontinuous. 
This motivates the following problem:

\begin{problem}\label{problem:properness}
Find an effective criterion for determining whether a given action is proper (see Definition \ref{definition:properaction} for the definition of proper actions).
\end{problem}

A systematic study of this problem was initiated by Kobayashi \cite{Kobayashi89,Kobayashi96} in the context
of discontinuous groups and Clifford--Klein forms of homogeneous spaces.
He considered the following setting:
let $G$ be a locally compact Hausdorff group, $L,H$ closed subgroups of $G$ and $\rho$ the natural $L$-action on the homogeneous space $G/H$, that is, 
\[
    \rho:L\times G/H\to G/H, \quad (\ell,gH)\mapsto \ell gH.
\]
One of the basic ideas in his approach is to regard the properness of $\rho$ as a relationship between the two subgroups $L$ and $H$
inside $G$ (see Proposition \ref{proposition:properbyside}).
When $G$ is a reductive Lie group, effective criteria for properness were established by Kobayashi \cite{Kobayashi89,Kobayashi96} and Benoist \cite{Benoist96}.
Further, in \cite{Kobayashi92Fuji}, Kobayashi introduced the (CI) condition and investigated its relationship with properness.
Here, the $(\mathrm{CI})$ condition means that the isotropy at every point in $G/H$ is compact.
The properness of $\rho$ always implies the $(\mathrm{CI})$ condition.
Thus, the following problem naturally arises.

\begin{question}\label{question:ci}
Under what assumptions does the $(\mathrm{CI})$ condition imply that $\rho$ is proper?
\end{question}

He gave an affirmative answer to the above question when $G:=\GL(2,\R)\ltimes \R^2, H:=\GL(2,\R)$ and $L$ is a connected closed subgroup of $G$.

Motivated by this problem, the case where $G$ is a solvable group has also been studied.
Lipsman \cite{Lipsman1995proper} established the same implication in the case where $G:=\rmN(3)\ltimes \R^3, H:=\rmN(3)$ and $L$ is a connected closed subgroup of $G$.
He also conjectured that Question~\ref{question:ci} has an affirmative answer when $G$ is a 1-connected nilpotent Lie group and $L,H$ are connected closed subgroups of $G$.
Lipsman's conjecture motivated a series of works on proper actions where $G$ is a nilpotent or solvable Lie group.
Nasrin \cite{Nasrin2001twostep} proved the conjecture for two-step nilpotent Lie groups.
Further affirmative results for three-step nilpotent groups were obtained independently by Baklouti--Khlif \cite{BakloutiKhlif2005ExpSolv} and Yoshino \cite{Yoshino2007threestep}.
On the other hand, Yoshino \cite{Yoshino2005counterLips} constructed counterexamples of Lipsman's conjecture, showing that it fails in general already in the four-step nilpotent setting.
Further investigations of proper actions on exponential solvable
homogeneous spaces were carried out by Baklouti--Khlif \cite{BakloutiKhlif2005ExpSolv,BakloutiKhlif2007weak}.

More recently, Maeta \cite{Maeta2023fourIJM} developed cocompact-properness criteria adapted to the case where $G$ is a 1-connected solvable Lie group, and $L,H$ are connected subgroups of $G$ (see Section 3 in \cite{Maeta2023fourIJM}).

In this paper, we investigate Problem~\ref{problem:properness} when $G$ is the group of invertible upper triangular matrices, or more generally, a subgroup of it.
Furthermore, we will approach the problem of determining the properness of $\rho$ from a coarse-geometric perspective.
For this reason, we will also touch on coarse geometry below.

The development of large-scale geometry was strongly influenced by Gromov~\cite{Gromov1987,Gromov1993}.
Building on this, Roe \cite{Roe1993,Roe2003LectureCoarse} formulated coarse geometry.
A coarse space is a generalization of a metric space that retains its large-scale properties.
In this setting, a natural large-scale analogue of disjointness is
\emph{asymptotic disjointness}.
It is studied in several settings  by Dranishnikov, Grzegrzolka--Siegert, Kalantari--Honari, and Protasov \cite{Dranishnikov2001,GrzegrzolkaSiegertproximity2019, KalantariHonari2015, Protasov2003}.

Recently, the properness problem for homogeneous spaces has been related
directly to this notion.
For a locally compact group $G$, one may consider the coarse structure
generated by left and right multiplications, which we call the
LR-coarse structure (see Example \ref{example:coarse}).
For closed subgroups $L,H$ of $G$, the properness of the $L$-action on
$G/H$ can then be characterized by the asymptotic disjointness of $L$ and
$H$ in $G$ equipped with the LR-coarse structure
\cite{MiyajiNagayaOgawaOkuda2026,NagayaOgawaOkuda2025} (see Example \ref{example:properad}).
This gives a large-scale interpretation of the properness relation between
two subgroups.

With this coarse-geometric interpretation in mind, we now state the main results of this paper as below.

\subsection{Main Theorems}

To introduce our main results, we fix the following notation:

\begin{itemize}
    \item To simplify notation, for any $n\in\mathbb{N}$ and any subset $S\subset \text{M}(n,\R)$, we use the following same notation:
    \[
    \|-\|:S\to\R_{\geq 0},\quad  X:=(x_{ij})_{1\leq i,j\leq n}\mapsto \|X\|:=\max\{|x_{ij}|\mid 1\leq i,j\leq n\}.
    \]
    \item For $n\in \mathbb{Z}_{\geq 1}$ and $q\in [n]:=\{1,\ldots,n\}$, let
    $N:=\binom{n}{q}$.
    We denote by
    \[
    \bigwedge^q:\GL(n,\mathbb{R})\to \GL(N,\mathbb{R})
    \]
    the $q$-th exterior power representation.
    With respect to the standard basis of $\bigwedge^q\mathbb{R}^n$ indexed by
    $q$-element subsets of $[n]$, the $(I,J)$-entry of $\bigwedge^q A$ is
    $\det(A_{IJ})$ (see Subsection \ref{subsection:grouphom} for more details).
    Here, $N:=\binom{n}{q}$.
    \item For $n\in \Z_{\geq 1}$ and $1\leq i\leq j\leq n$,  
    \[
    \Phi_{ij} :\rmB(n) \to \rmB(j-i+1),\quad (x_{st})_{1\leq s, t \leq n}\mapsto (x_{st})_{i\leq s, t \leq j}.
    \]
    \item For $n\in \Z_{\geq 1}$, 
    we define the following:
    \begin{align*}
        \Lambda_0(n)&:=\{(i,j,q,\varepsilon)\in [n]^3\times \{\pm 1\}\mid q=1,\varepsilon=1\text{ and } \left(i=j \text{ or }(i,j)=(1,n)\right)\}, \\
        \Lambda_1(n)&:=\{(i,j,q,\varepsilon)\in [n]^3\times \{\pm 1\}\mid i\leq j,1\leq q \leq j-i+1\}.
    \end{align*}
    Further, for a nonempty subset $\Lambda\subset \Lambda_1(n)$, put \[
        \mu_\Lambda : \rmB(n)\to \R^{\Lambda}, \quad X:=(x_{ij})_{i,j\in [n]}\mapsto (\log \|\bigwedge^q\Phi_{ij}(X^\varepsilon)\|)_{(i,j,q,\varepsilon)\in \Lambda}.
        \]
    In particular, let us denote $\mu_{\Lambda_0}$ by $\mu_0$, and $\mu_{\{\lambda\}}$ by $\mu_\lambda$ for each $\lambda\in \Lambda_1(n)$.
\end{itemize}

We now state our two main theorems as below:

\begin{maintheorem}\label{maintheorem:mu0}
    Fix $n\in \Z_{\geq 1}$ and a subset $\Lambda\subset \Lambda_1(n)$ with $\Lambda_0(n)\subset \Lambda$.
    Let $L,H$ be closed subgroups of $\rmB(n)$.
    Assume that the intersection $\mu_{\Lambda}(L)\cap \bar{B}_R[\mu_\Lambda(H)]$ is bounded in $\R^{\Lambda}$ for each $R\geq 0$.
    Then the $L$-action on $\rmB(n)/H$ is proper. 
    Here, $\bar{B}_R[\mu_\Lambda(H)]$ is the closed $R$-neighborhood of $\mu_\Lambda(H)$.
\end{maintheorem}

Furthermore, fix $n\geq 2$, and consider the following case:
\[
    H:=\{X\in \rmB(n)\mid x_{in}=0 \text{ for all } 1\leq i<n \}.
\]
Then the homogeneous space $\rmB(n)/H$ can be regarded as $\R^{n-1}$.
Let $\rho$ be the $\rmB(n)$-action on $\R^{n-1}$ as below:
\begin{align*}
    \rmB(n)\times \R^{n-1}&\to \R^{n-1}\\
    \left(
    \begin{pmatrix}
        A &v \\
        0 & \lambda
    \end{pmatrix},w
    \right)
    &\mapsto  \frac{Aw+v}{\lambda}.
\end{align*} 
Here, $A\in \rmB(n-1),v\in \R^{n-1}$ and $\lambda\in \R^\times$.
Define the following:
\[
    \Theta(n):=\{(i,q,\varepsilon)\in [n]\times [n]\times \{\pm 1\}\mid q\leq n-i\}.
\]
For a nonempty subset $\Theta\subset \Theta(n)$, put
\[
    \varphi_{\Theta}:\rmB(n)\to \R^{\Theta}, \quad X\mapsto 
    \left(
    \log\left(
    \frac{\|\bigwedge^q \Phi_{in}(X^\varepsilon)\|}{\max\{\|\bigwedge^q\Phi_{in-1}(X^\varepsilon)\|,\|\bigwedge^{q-1}\Phi_{in-1}(X^\varepsilon)\||x_{nn}^\varepsilon|\}}
    \right)
    \right)_{(i,q,\varepsilon)\in \Theta}.
\]
Then the following main theorem holds:

\begin{maintheorem}     
    Let $L$ be a closed subgroup of $\rmB(n)$ and $\Theta$ a nonempty subset of $\Theta(n)$.
    Assume that the restriction $\varphi_{\Theta}|_L$ on $L$ is proper.
    Then the $L$-action $\rho|_L$ on $\R^{n-1}\cong \rmB(n)/H$ is proper.
\end{maintheorem}

Further, we discuss the case where $G$ is a closed subgroup of $\rmB(n)$ and $L,H$ are closed subgroups of $G$, with particular attention to the case $G=\rmN(n)$ in Section~\ref{section:subgroups}. 
Here, $\rmN(n)$ denotes the group of upper triangular matrices whose diagonal entries are all equal to $1$.

\section{Preliminaries on proper actions}

In this section, we recall some definitions and properties concerning proper actions and group homomorphisms that will be used in the subsequent sections.

\subsection{Proper actions}

In this section, we recall the definitions of proper maps and proper actions, together with some characterizations for properness of actions.

\begin{definition}
    Let $S$ and $T$ be Hausdorff spaces, $f$ be a continuous map from $S$ to $T$.
    The map $f$ is called \emph{topologically proper} if $f^{-1}(C)$ is compact in $S$ for any compact subset $C\subset T$. 
\end{definition}

\begin{definition}\label{definition:properaction}
    Let $L$ be a locally compact Hausdorff group, $S$ a locally compact Hausdorff space, $\rho :L\times S\rightarrow S$ a continuous action. 
    The action $\rho$ is said to be \emph{proper} if the map
    \[
    \Psi : L \times S \to S \times S, \quad (g, x) \mapsto (x, gx)
    \]
    is topologically proper.
\end{definition}

Under the settings below, the properness of group actions can be restated as follows:

\begin{proposition}[cf.~\cite{Kobayashi96}]\label{proposition:properbyside}
    Let $G$ be a locally compact Hausdorff group, $L$ and $H$ closed subgroups of $G$.
    The following six conditions are equivalent:
    \begin{enumerate}
        \item The natural $L$-action on the homogeneous space $G/H$ is proper.
        \item The natural $H$-action on the homogeneous space $G/L$ is proper.
        \item The diagonal $G$-action on the space $G/H\times G/L$ is proper.
        \item\label{proposition:properbyside:H} The intersection $L\cap CHC^{-1}$ is relatively compact in $G$ for any compact subset $C$ of $G$.
        \item\label{proposition:properbyside:L} The intersection $CLC^{-1}\cap H$ is relatively compact in $G$ for any compact subset $C$ of $G$.
        \item\label{proposition:properbyside:HL} The intersection $CHC^{-1}\cap CLC^{-1}$ is relatively compact in $G$ for any compact subset $C$ of $G$.
    \end{enumerate}    
\end{proposition}

Note that the three conditions \eqref{proposition:properbyside:H}, \eqref{proposition:properbyside:L} and \eqref{proposition:properbyside:HL} in the above proposition make sense without assuming that $L$ and $H$ are closed subgroups.

\subsection{Group morphisms}\label{subsection:grouphom}

In this section, we record two group homomorphisms as below.

For each $n,q \in \Z_{\geq 1}$ with $q\leq n$, we fix the following notations throughout this paper:
\begin{align*}
[n]&:=\{1,\cdots, n\}, \\
\binom{[n]}{q}&:=\{I:[q] \to [n]\mid I  \text{ is strictly monotonically increasing}\}.
\end{align*}
We equip $\binom{[n]}{q}$ with the lexicographic order defined by:
\begin{align*}
I\leq_{\text{lex}} J &:\Leftrightarrow I=J \text{ or }\exists a\in [q] \text{ such that } I(a)< J(a)  \text{ and } I(b)=J(b) \, (\text{for all } 1\leq b<a),\\
\end{align*}
and for $X\in \mathrm{M}(n,\R)$, 
\begin{align*}
X_{IJ}&:=(x_{I(i)J(j)})_{i,j\in [q]} \in \text{M}(q,\R).
\end{align*}
We write the bijective strictly monotonically increasing map from $\left[\binom{n}{q}\right]$ to $\binom{[n]}{q}$ for $\mathbb{I}_{n,q}$.

We now recall the following proposition: 

\begin{proposition}\label{proposition:outergrouphom}[cf.~\cite{HornJohnson2013}]
    For each $n,q\in \Z_{\geq 1}$ with $q\leq n$,
    the following map $\bigwedge^q$ is well-defined and a continuous group homomorphism:
    \[
        \bigwedge^q : \GL (n,\R) \to \GL (N,\R), \quad X\mapsto (\det X_{\mathbb{I}_{n,q}(s)\mathbb{I}_{n,q}(t)})_{s,t\in \left[\binom{n}{q}\right]}.
    \]
    Here, $N:=\binom{n}{q}$.
\end{proposition}

Throughout this paper, for $n\in \Z_{\geq 1}$, fix the notation as below:
\[
    \rmB(n):=\{X\in \GL(n,\R)\mid x_{ij}=0 \text{ for all }i>j\}.
\]

We will also use the following elementary homomorphisms in Section \ref{section:properaction}.

\begin{proposition}\label{proposition:phiij}
    For each $n\in \Z_{\geq 1}$ and $i,j\in [n]$ with $i\leq j$,
    the following map is a continuous group homomorphism:
    \begin{align*}
        \Phi_{ij}: \rmB(n) &\to \rmB(j-i+1), \\
        \begin{pmatrix}
            x_{11} & x_{12} &\cdots & x_{1n}\\
            0 & \ddots &\ddots & \vdots\\
            \vdots & \ddots & \ddots & x_{n-1 n }\\
            0 & \cdots & 0 &x_{nn}
        \end{pmatrix} &\mapsto 
        \begin{pmatrix}
            x_{ii} & x_{ii+1} &\cdots & x_{ij}\\
            0 & \ddots &\ddots & \vdots\\
            \vdots & \ddots & \ddots & x_{j-1 j }\\
            0 & \cdots & 0 &x_{jj}
        \end{pmatrix}.
    \end{align*}
\end{proposition}

One can see that the above proposition holds.
In fact, we have:
\begin{align*}
    \Phi_{ij}(XY)_{st}&=(XY)_{(s+i-1)(t+i-1)}=\sum_{1\leq k\leq n}x_{(s+i-1)k}y_{k(t+i-1)}\\
    &= \sum_{s+i-1\leq k\leq t+i-1}x_{(s+i-1)k}y_{k(t+i-1)}\\
    &= \sum_{s\leq k\leq t}(\Phi_{ij}(X))_{sk}(\Phi_{ij}(Y))_{kt}
    =(\Phi_{ij}(X)\Phi_{ij}(Y))_{st}.
\end{align*}

\section{Preliminaries on bornological spaces and coarse spaces}

In this section, we recall the definition and some properties of bornological spaces and  coarse spaces.

\subsection{Notations on relations}\label{subsection:notation}

In this section, we fix some terminology and notation.
The notation defined in this section will be used throughout this paper.

For each set $S$, 
we write $\mathcal{P}(S)$ for the family of all subsets of $S$, $\diag(S)$ for the diagonal set of $S$, and $\# S$ for the cardinality of $S$.
For sets $S$ and $T$, a subset of $T \times S$ is called a \emph{relation} from $S$ to $T$.
For a relation $E$ from $S$ to $T$, we define the relation \(E^\dagger\) from $T$ to $S$ by
\[
E^\dagger := \{ (x,y) \mid (y,x) \in E \} \subset S \times T.
\]

For sets $S,T$ and $U$, and for
$E \in \mathcal{P}(T \times S)$ and $F \in \mathcal{P}(U \times T)$, we define their composition
$F \circ E \in \mathcal{P}(U \times S)$ by
\[
F \circ E
:= \{ (z,x) \in U \times S
\mid \text{there exists } y \in T 
\text{ such that } (z,y) \in F \text{ and } (y,x) \in E \}.
\]
We also write $\diag(S) := \{ (x,x) \mid x \in S \} \subset S \times S$.

For each relation $E$ from $S$ to $T$ and each subset $A$ of $S$, we shall define the subset $E[A]$ of $T$ by 
\[
E[A] := \{ y \in T \mid \text{there exists } a \in A \text{ such that } (y,a) \in E \}.
\]
When $S = T$, the subset $E[A]$ is called the \emph{$E$-neighborhood} of $A$ by $E$.
In particular, for $E\in \mathcal{P}(T\times S)$ and $p\in S$, put $E[p]:=E[\{p\}]$.

Throughout this paper, 
a map $f : S \rightarrow T$ is regarded as a relation from $S$ to $T$, that is, 
\[
f = \{ (f(x),x) \mid x \in S \} \subset T \times S.
\]

Next, we define the \emph{adjoint operator} as below:

\begin{definition}[cf.~\cite{MiyajiNagayaOgawaOkuda2026}]
    Let $S$ and $T$ be both sets. 
    For each relation $\eta \subset T \times S$, 
    we define the operator 
    \[
    \Ad_\eta : \mathcal{P}(S \times S) \rightarrow \mathcal{P}(T \times T), ~ E \mapsto \eta \circ E \circ \eta^{\dagger}.
    \]
\end{definition}

The following proposition is a basic property of the adjoint operator:

\begin{proposition}
    Let $S,T$ and $U$ be sets.
    For each $\eta\in \mathcal{P}(T\times S)$ and $\theta \in \mathcal{P}(U\times T)$, the equation 
    $\Ad_{\theta \circ \eta}=\Ad_\theta\circ \Ad_\eta$ holds.  
\end{proposition}

Note that for each map $f : S \rightarrow T$ and each $E \in \mathcal{P}(S \times S)$ and $F \in \mathcal{P}(T \times T)$, we have
\begin{align*}    
\Ad_f(E) &= \{ (f(x_1),f(x_2)) \in T\times T\mid (x_1,x_2) \in E \} \in \mathcal{P}(T \times T), \\
\Ad_{f^\dagger}(F) &= \{ (x_1,x_2) \in S \times S \mid (f(x_1),f(x_2)) \in F \}  \in \mathcal{P}(S \times S),
\end{align*}
and $\diag(S)\subset f^\dagger \circ f$.

\subsection{Bornological spaces}
In this subsection, we recall the definition of bornological spaces (see \cite{HogbeHenri77}) and some basic propositions.

Bornological spaces are defined as below:

\begin{definition}\label{def:bornology}
    Let $S$ be a set.
    A family $\mathcal{B}\subset \mathcal{P}(S)$ of subsets of $S$ is called a \emph{bornology} on $S$ if it satisfies the following three conditions. 
    \begin{enumerate} 
        \item $\mathcal{B}$ is a cover of $S$ (i.e.~ $\bigcup_{B\in \mathcal{B}}B=S$).
        \item $B_1 ,B_2\in \mathcal{B}$ implies $B_1\cup B_2\in \mathcal{B}$. 
        \item $B\in \mathcal{B}$ and $B'\subset B$ imply $B'\in \mathcal{B}$. 
    \end{enumerate}
    The pair $(S,\mathcal{B})$ is called a \emph{bornological space}.
\end{definition}

A family $\mathcal{B}_0\subset \mathcal{P}(S)$ is called a \emph{bornological base} on $S$ if the following conditions hold: \begin{enumerate} \item $\bigcup_{B\in\mathcal{B}_0}B=S$; \item for any $B_1,B_2\in\mathcal{B}_0$, there exists $B_0\in\mathcal{B}_0$ such that $B_1\cup B_2\subset B_0$. \end{enumerate} In this case, we write \[ \langle\mathcal{B}_0\rangle :=\{B\subset S\mid B\subset B_0 \text{ for some } B_0\in\mathcal{B}_0\}, \] which forms a bornology on $S$.

Next, we recall two examples of bornological spaces:

\begin{example}\label{example:bornology}
    We have the following:
    \begin{enumerate}
        \item Let $(S,d)$ be a metric space. Then the family $\mathcal{B}_d(S)$ of all metrically bounded subsets forms a bornology on $S$. 
        It is called the \emph{metric bornology}. 
        \item Let $S$ be a Hausdorff space. Then the family $\mathcal{B}_\mathrm{cpt}(S)$ of all relatively compact subsets forms a bornology on $S$. 
        It is called the \emph{compact bornology}.
    \end{enumerate}
\end{example}

Note that for a metric space $(S,d)$, the space is Heine-Borel (i.e.~ every bounded closed subset is compact) if and only if the equality $\mathcal{B}_d(S)=\mathcal{B}_\mathrm{cpt}(S)$ holds.

In the rest of this section, let $(S,\mathcal{B}_S)$ and  $(T,\mathcal{B}_T)$ be bornological spaces.

We shall define the terminologies of maps between bornological spaces as below:

\begin{definition}\label{def:bmaps}
    Let $f$ be a map from $S$ to $T$.
    \begin{enumerate}
        \item The map $f$ is called \emph{bornological} if $f(B)\in \mathcal{B}_T$ for every $B\in \mathcal{B}_S$.
        \item The map $f$ is called  \emph{bornologically proper} (abbreviated as B-proper in the sequel) if $f^\dagger[D]=f^{-1}(D)\in \mathcal{B}_S$ for every $D\in \mathcal{B}_T$.
    \end{enumerate}
\end{definition}

\begin{remark}\label{remark:btop}
    Suppose that $S$ and $T$ are Hausdorff spaces equipped with their compact bornologies.
    Then, note that every continuous map $f:S\to T$ is bornological.
    Furthermore, for each continuous map $f:S\to T$, 
    $f$ is topologically proper if and only if $f$ is B-proper.
\end{remark}

Let us recall the definitions of direct product of bornologies and induction of a bornology.

\begin{definition}\label{def:induction}
    We define the following two bornologies:
    \begin{enumerate}
        \item For a family of bornological spaces $\{(S_\lambda, \mathcal{B}_\lambda)\}_{\lambda\in \Lambda}$, the set-theoretical direct product $\Pi_{\lambda\in \Lambda}\mathcal{B}_\lambda$ forms a bornological base on $\Pi_{\lambda\in \Lambda} S_\lambda$.
        We call $\langle \Pi_{\lambda\in \Lambda}\mathcal{B}_\lambda\rangle$ the \emph{direct product bornology} on $\Pi_{\lambda\in \Lambda} S_\lambda$.
        \item Let $P$ be a set and $f$ a map from $P$ to the bornological space $S$.
        Then the family $f^*\mathcal{B}_S:=\{f^{-1}(B)\subset P\mid B\in \mathcal{B}_S\}$ forms a bornological base on $P$.
        We call $\langle f^*\mathcal{B}_S \rangle $ the \emph{inverse image bornology} of $\mathcal{B}_S$ by the map $f$. 
        In particular, if $P$ is a subset of $S$ and the map $f$ is the inclusion map, the bornology $f^*\mathcal{B}_S =\langle f^*\mathcal{B}_S \rangle $ is called the \emph{subbornology} of $\mathcal{B}_S$.
    \end{enumerate}
\end{definition}

In the remainder of this subsection, let us focus on compact  bornologies on Hausdorff spaces.
Throughout this paper, for a set $S$ and a subset $P\subset S$, 
let us write the inclusion map from $P$ to $S$ for $\iota_P$.

\begin{proposition}\label{proposition:closedeqqborno}[cf.~Proposition 5.6 in \cite{Nagaya2026}]
    Let $S$ be a locally compact Hausdorff space, and 
    $P$ a subset of $S$.
    Then the following conditions on $P$ are equivalent:
    \begin{enumerate}
        \item\label{proposition:closedeqqborno:closed}  The set $P$ is closed in $S$.
        \item\label{proposition:closedeqqborno:bornoequation}  The equality $\iota_P^*\mathcal{B}_\mathrm{cpt}(S)=\mathcal{B}_\mathrm{cpt}(P)$ holds. 
    \end{enumerate}
\end{proposition}

In the above setting, 
note that the inclusion $\iota_P^*\mathcal{B}_\mathrm{cpt}(S)\supset \mathcal{B}_\mathrm{cpt}(P)$ always holds in general.

Further, one can see that the following proposition holds:

\begin{proposition}\label{propsition:properB-proper}
    Let $S$ and $T$ be Hausdorff spaces and $f$ be a map from $S$ to $T$.
    Assume that $f$ is continuous.
    Then the map $f$ is topologically proper if and only if $f$ is B-proper from $(S,\mathcal{B}_\text{cpt}(S))$ to $(T,\mathcal{B}_\text{cpt}(T))$.
\end{proposition}

\subsection{Preliminaries on coarse spaces}

In this subsection, let us recall definitions of coarse spaces and some properties of them.

\begin{definition}\label{definition:coarse}
    Let $S$ be a set.
    A family $\mathcal{E}\subset \mathcal{P}(S\times S)$ is called a \emph{coarse structure} on $S$ if it satisfies the following five conditions (see Section \ref{subsection:notation} for notation):
    \begin{enumerate}
        \item $\diag(S) \in \mathcal{E}$.
        \item $E \in \mathcal{E}$ implies $E^\dagger \in \mathcal{E}$.
        \item $E_1, E_2 \in \mathcal{E}$ implies $E_1 \cup E_2 \in \mathcal{E}$.
        \item $E_1, E_2 \in \mathcal{E}$ implies $E_1\circ E_2\in \mathcal{E}$.
        \item $E \in \mathcal{E}$ implies $E' \in \mathcal{E}$ for each $E' \subset E$.
    \end{enumerate}
    The pair $(S, \mathcal{E})$ is called a \emph{coarse space} and 
    each element of $\mathcal{E}$ is referred to as an \emph{coarse entourage} or simply an \emph{entourage}.
    An entourage $E\in \mathcal{E}$ is called \emph{symmetric} if the equality $E=E^\dagger$ holds.
\end{definition}

A subset $B$ of $S$ is called \emph{coarsely bounded} or simply \emph{bounded} in $(S,\mathcal{E})$ 
if $B \times B \in \mathcal{E}$.
Note that the following basic three properties of boundedness hold:
\begin{itemize}
    \item When $S \neq \emptyset$, 
        a subset $B$ of $S$ is bounded if and only if $B \subset E[p]$ for some $p\in S$ and $E \in \mathcal{E}$ (see \cite[Proposition 2.16]{Roe2003LectureCoarse}).
    \item any subset of a bounded set is also bounded,
    \item for any bounded set $B$ and any entourage $E$, the $E$-neighborhood $E[B]$ of $B$ is also bounded.
\end{itemize}

We put 
\[
\mathcal{B}_\mathcal{E}:=\{\cup_{i=1}^nB_i\subset S \mid n\in \Z_{\geq 1}, \{B_i\}_{i=1}^n\subset \mathcal{P}(S) \text{ is a family of bounded sets} \}.
\]

It is worth emphasizing that $\mathcal{B}_\mathcal{E}$ forms a bornology on $S$.
The bornology $\mathcal{B}_\mathcal{E}$ is called \emph{the bornology induced by the coarse structure} $\mathcal{E}$.

A coarse space $(S,\mathcal{E})$ is called  \emph{coarsely connected} if 
for any pair of points $p,q \in S$, 
there exists an entourage $E \in \mathcal{E}$ such that $q \in E[p]$.
In the setting where $(S,\mathcal{E})$ is coarsely connected, note that $\mathcal{B}_\mathcal{E}$ coincides with the family of all bounded subsets.

We call a subset $\mathcal{E}_0$ of $\mathcal{P}(S\times S)$  a \emph{base of a coarse structure} if it satisfies the following four conditions (cf.~\cite{DikranjanZavaCatCoarseGp2020}):

\begin{enumerate}
    \item[(i)'] There exists $E_0 \in \mathcal{E}_0$ such that  $\diag (S) \subset E_0$.
    \item[(ii)'] For each $E_1 \in \mathcal{E}_0$, there exists $E_0 \in \mathcal{E}_0$ such that $E_1^\dagger \subset E_0$.
    \item[(iii)'] For each $E_1,E_2 \in \mathcal{E}_0$, 
    there exists $E_0 \in \mathcal{E}_0$ such that $E_1 \cup E_2 \subset E_0$.
    \item[(iv)'] For each $E_1,E_2 \in \mathcal{E}_0$, 
    there exists $E_0 \in \mathcal{E}_0$ such that $E_1 \circ E_2 \subset E_0$.
\end{enumerate}

For a base of a coarse structure $\mathcal{E}_0$, the following family $\langle \mathcal{E}_0\rangle$ forms a coarse structure: 
\[
    \langle \mathcal{E}_0\rangle:=\{E\subset S\times S\mid \text{there exists } E_0\in \mathcal{E}_0 \text{ such that }E\subset E_0\}.
\]

\begin{example}\label{example:coarse}
Here are two examples of coarse spaces.
\begin{enumerate}

\item Let $(S,d)$ be a metric space. For each $r\geq 0$, we put 
        \[
        E_r:=\{(x,y)\in S\times S\mid d(x,y)\leq r\}. 
        \]
        Then the family $\mathcal{E}_0^d:=\{E_r\mid r\geq 0\}$ forms a coarse base and the coarse structure $\mathcal{E}^d(S)$ induced by $\mathcal{E}_0^d$ is called the \emph{bounded coarse structure} on $S$.
        For simplicity, we sometimes denote $\mathcal{E}^d(S)$ by $\mathcal{E}^d$.
        In this setting, a subset $B$ of $S$ is bounded for $\mathcal{E}^d(S)$ if and only if $B$ is bounded with respect to the metric $d$ on $S$. 
\item Let $S$ be a set, $G$ a locally compact Hausdorff group and     $\rho$ a $G$-action on $S$.
    For each compact subset $C\subset G$,  we put 
    \[
    E^\rho_C:=\{(x,y)\in S\times S\mid x\in C\cdot_\rho y\}.
    \]
    Then the family $\mathcal{E}_0^\rho:=\{E^\rho_C\mid C\subset G \text{ : compact}\}$ forms a coarse base, and the coarse structure $\mathcal{E}^\rho$ induced by this is called the \emph{coarse structure induced by} $\rho$.
    Furthermore, the following basic three properties hold:
    \begin{itemize}
        \item Let us put:
        \[
        \mathcal{K}(G):=\{C\subset G\mid C \text{ :compact}, C=C^{-1}, e\in C\}.
        \]
        Then the family $\{E_C^\rho \mid C\in \mathcal{K}(G)\}$ is also a coarse base of $\mathcal{E}^\rho$ on $S$.
        \item The coarse space $(S,\mathcal{E}^\rho)$ is coarsely connected, if and only if the action $\rho$ is transitive.
        \item Assume that $S$ is Hausdorff space, and $\rho$ is continuous. Then, the induced bornology $\mathcal{B}_{\mathcal{E}^\rho}$ is a subset of $\mathcal{B}_{\text{cpt}}(S)$.
    \end{itemize}
    In particular, we focus on the following three actions as below:
    \begin{align*}
        \rho^\text{L}&: G\times G\to G, \quad (g,h)\mapsto gh,\\
        \rho^\text{R}&: G\times G\to G, \quad (g,h)\mapsto hg^{-1},\\
        \rho^\text{LR}&: (G\times G)\times G\to G, \quad ((g,\ell),h)\mapsto gh\ell^{-1}.
    \end{align*}
    The coarse structures on $G$ induced by $\rho^\mathrm{L}$, $\rho^\mathrm{R}$, and $\rho^\mathrm{LR}$ are called the \emph{L-coarse structure}, \emph{R-coarse structure}, and \emph{LR-coarse structure} on $G$, respectively, and are denoted by $\mathcal{E}^\mathrm{L}$, $\mathcal{E}^\mathrm{R}$, and $\mathcal{E}^\mathrm{LR}$, respectively.
    Further, for a compact set $C\subset G$, put
    $E_C^\text{L}:=E_C^{\rho^\text{L}}$,
    $E_C^\text{R}:=E_C^{\rho^\text{R}}$.
    and $E_C^\text{LR}:=E_{C\times C}^{\rho^\text{LR}}$.
    It immediately holds that the family $\{E_C^\text{LR}\mid C\in \mathcal {K}(G)\}$ forms a coarse base of $\mathcal{E}^\text{LR}$.
    Furthermore, $\mathcal{B}_{\mathcal{E}^\text{L}}=\mathcal{B}_{\mathcal{E}^\text{R}}=\mathcal{B}_{\mathcal{E}^\text{LR}}=\mathcal{B}_{\text{cpt}}(G)$.
    Note that, when $G$ is abelian, then $\mathcal{E}^\mathrm{L}= \mathcal{E}^\mathrm{R}=\mathcal{E}^\mathrm{LR}$.
\end{enumerate}
\end{example}

Next we shall define the product of coarse spaces as below:

\begin{definition}
    Let $\{(S_\lambda,\mathcal{E}_\lambda)\}_{\lambda \in \Lambda}$ be a family of coarse spaces.
    Fix a notation:
    \[
    \otimes_{\lambda\in \Lambda}\mathcal{E}_\lambda :=\{E\subset \Pi_{\lambda \in \Lambda} S_\lambda\times \Pi_{\lambda \in \Lambda} S_\lambda\mid \text{for all }\lambda \in \Lambda, \Ad_{\pi_{\lambda}}(E)\in \mathcal{E}_\lambda \}.
    \]
    Then this forms a coarse structure on $\Pi_{\lambda \in \Lambda} S_\lambda$, and $(\Pi_{\lambda \in \Lambda} S_\lambda,\otimes_{\lambda\in \Lambda}\mathcal{E}_\lambda)$ is called the \emph{product coarse structure} of $\{(S_\lambda,\mathcal{E}_\lambda)\}_{\lambda \in \Lambda}$.
    Here, $\pi_{\lambda}$ is the projection to $S_{\lambda}$ for each $\lambda\in \Lambda$.
\end{definition}

It is remarkable that, 
for a family of coarsely connected coarse spaces $\{(S_\lambda,\mathcal{E}_\lambda)\}_{\lambda \in \Lambda}$,
the bornology induced by the coarse structure $\otimes_{\lambda\in \Lambda}\mathcal{E}_\lambda$ coincides with the product bornology $\Pi_{\lambda \in \Lambda}\mathcal{B}_{\mathcal{E}_\lambda}$.

\begin{example}
    Fix $n\in \Z_{\geq 1}$ and let $\{(S_i,d_i)\mid i\in [n]\}$ be a family of metric spaces.
    Then, the product coarse structure $\otimes \mathcal{E}^{d_i}$ of the bounded coarse structures coincides with the bounded coarse structure induced by any of the standard \(\ell^p\)-product metrics on \(\prod_{i=1}^n X_i\), \(1\leq p\leq\infty\). Indeed, the product coarse structure is induced by  
    $d_\infty$ (cf.~\cite{BanakhBanakh2022}),
    and, for \(1\leq p<\infty\),
    $d_\infty\leq d_p\leq n^{1/p}d_\infty$.
    
    Hence all these metrics induce the same bounded coarse structure.
\end{example}

\subsection{Maps between coarse spaces}

In this section, 
mainly following \cite{DikranjanZavaCatCoarseGp2020, Roe2003LectureCoarse}, 
we set up our terminologies for maps between coarse spaces.
To define the coarse equivalence, we prepare the following:

\begin{definition}
    Let $S$ be a set, $T$ a coarse space, and $f,g$ maps from $S$ to $T$.
    The maps $f$ and $g$ are \emph{close} (denoted by $f\sim g$) if the set 
    \[
    \Xi_{f,g}:=f\circ g^{\dagger}=\{(f(x),g(x))\in T\times T \mid x\in S\}
    \]
    is an entourage of $T$.
\end{definition}

One can see that closeness defines an equivalence relation on the set of all maps
from a fixed set $S$ to a coarse space $T$.

\begin{example}\label{example:closetoR}
    Let $S$ be a nonempty set, and $f,g$ functions from $S$ to $\R$.
    Consider $\R$ is equipped with the bounded coarse structure.
    The maps $f$ and $g$ are close if and only if $\sup_{x\in S}|f(x)-g(x)|<\infty$.
    Let $f_1,g_1,f_2$ and $g_2$ be functions from $S$ to $\R$ with $f_1\sim g_1$ and $f_2\sim g_2$.
    One can see that the following basic four properties of closeness hold:
    \begin{itemize}
        \item $f_1+f_2\sim g_1+g_2$.
        \item $f_1-f_2\sim g_1-g_2$.
        \item $\max\{f_1,f_2\}\sim \max\{g_1,g_2\}$.
        \item $\min\{f_1,f_2\}\sim \min\{g_1,g_2\}$.
    \end{itemize}
\end{example}

Next, we fix some terminology of maps between coarse spaces as below:

\begin{definition}\label{definition:coarsemaps}
Let $(S,\mathcal{E})$ and $(T,\mathcal{F})$ be coarse spaces.
A map $f : S \rightarrow T$ is called: 
\begin{enumerate}
    \item\label{definition:coarsemaps:item:controlled} \emph{controlled} (cf.~\cite{BunkeEngel2020}) if, for each $E\in \mathcal{E}$, 
    \[
    \Ad_f(E) = \{ (f(x_1),f(x_2)) \mid (x_1,x_2) \in E \} \subset T \times T
    \]
    is an entourage on $T$, 
    \item\label{definition:coarsemaps:item:proper}  \emph{coarsely proper} (cf.~\cite{LeitnerVigolo2023}) if the set $f^\dagger[D]=f^{-1}(D)$ is bounded in $S$ for any bounded subset $D$ of $T$,  
    \item\label{definition:coarsemaps:item:coarse} \emph{coarse} (cf.~\cite{Roe2003LectureCoarse}) if it is controlled and coarsely proper,
    \item\label{definition:coarsemaps:item:effectivelyproper}  \emph{effectively proper} (cf.~\cite{DikranjanZavaCatCoarseGp2020}) if, for any $F\in \mathcal{F}$, the relation 
    \[
    \Ad_{f^\dagger}F = \{ (p,q) \mid (f(p),f(q)) \in F \} \subset S \times S
    \]
    is also an entourage on $S$,
    \item\label{definition:coarsemaps:item:embedding}   \emph{a coarse embedding} (cf.~\cite{DikranjanZavaCatCoarseGp2020}) if $f$ is controlled and effectively proper,
    \item\label{definition:coarsemaps:item:surjective}  \emph{coarsely surjective} (cf.~\cite{LeitnerVigolo2023}) if there exists $F\in \mathcal{F}$ such that $F[f(S)] = T$,
    \item\label{definition:coarsemaps:item:coarseeq}  \emph{a coarse equivalence} (cf.~\cite{DikranjanZava2017}) if $f$ is controlled and there exists a controlled map $h: T\to S$ such that $h\circ f\sim \id_S$ and $f\circ h\sim \id_T$.
\end{enumerate}
\end{definition}
In the above setting, note that the following basic three properties hold:
\begin{itemize}
    \item If $f$ is controlled, $f(B)$ is also bounded in $T$ for all bounded set $B$ in $S$. In particular, $f$ is a bornological map from $(S,\mathcal{B}_\mathcal{E})$ to $(T,\mathcal{B}_\mathcal{F})$.
    \item If $f$ is effectively proper, $f$ is coarsely proper.
    \item If $f$ is coarsely proper, then $f$ is a B-proper map from $(S,\mathcal{B}_\mathcal{E})$ to $(T,\mathcal{B}_\mathcal{F})$. Furthermore, if $(S,\mathcal{E})$ is coarsely connected, the converse claim also holds.
\end{itemize}

\begin{example}\label{example:controlledtoR}
    Let $(S,\mathcal{E})$ be a coarse space.
    the function $f:S\to \R$ is controlled if and only if 
    $\sup_{(x,y)\in E}|f(x)-f(y)|<\infty$ for any $E\in \mathcal{E}$.
    Here, $\R$ is equipped with the bounded coarse structure.
    Hence, for controlled functions $f,g$ from $S$ to $\R$,
    the four functions $f+g, f-g,\max\{f,g\}$ and $\min\{f,g\}$ are all controlled.
\end{example}

One can see that the following proposition holds:

\begin{proposition}\label{proposition:composition}
    The properties of being controlled, coarsely proper, coarse,
    effectively proper, a coarse embedding, and a coarse equivalence
    are preserved under composition.
\end{proposition}

\begin{remark}
    Coarse surjectivity is not preserved under composition in general.
    In fact, consider the following two maps:
    \[
    \Z\to \R,n\mapsto n, 
    \quad
    \R\to \R, x\mapsto \begin{cases}
        0 \quad (x\in \Z)\\
        x \quad (x\notin \Z)\\
    \end{cases}.
    \]
    The above two maps are coarsely surjective, but the composition of these maps is not coarsely surjective.
\end{remark}

For maps between coarse spaces, the following three propositions hold:

\begin{proposition}\label{proposition:closecomposition}
    Let $S,T,U$ be coarse spaces, $f,g$ maps from $S$ to $T$, and $\varphi,\psi$ controlled maps from $T$ to $U$.
    Assume $f\sim g$ and $\varphi\sim \psi$.
    Then maps $\varphi\circ f$ and $\psi\circ g$ are also close.
\end{proposition}

\begin{proposition}\label{proposition:preserveclose}
    The seven properties of maps defined in Definition \ref{definition:coarsemaps} are all preserved under closeness.
    More precisely, the following holds:
    Let $(S,\mathcal{E})$ and $(T,\mathcal{F})$ be coarse spaces, and $f,g$ maps from $S$ to $T$ with $f\sim g$.
    Assume that  $f$ has one of the properties defined in Definition \ref{definition:coarsemaps}.
    Then $g$ also has this property.
\end{proposition}

\begin{proposition}\label{proposition:products}
    Let $(S,\mathcal{E})$ be a coarse space, 
    $\{(T_\lambda,\mathcal{F}_\lambda)\}_{\lambda \in \Lambda}$ a family of coarse spaces, and 
    $\{f_\lambda : S\to T_\lambda\}_{\lambda\in \Lambda}$ a family of maps. 
    \begin{enumerate}
        \item\label{proposition:products:controlled} Assume that $f_\lambda$ is controlled for each $\lambda\in \Lambda$. Then, the map 
        \[
        (f_\lambda)_{\lambda \in \Lambda}:S\to \Pi_{\lambda \in \Lambda}T_{\lambda}, \, x\mapsto (f_{\lambda}(x))_{\lambda\in \Lambda}
        \]
        is also controlled.
        \item\label{proposition:products:coarse} Assume that  $f_\lambda$ is controlled for each $\lambda\in \Lambda$, and there exists $\lambda_0\in \Lambda$ such that $f_{\lambda_0}$ is coarse.
        Then, the map $(f_\lambda)_{\lambda \in \Lambda}:S\to \Pi_{\lambda \in \Lambda}T_{\lambda}$ is also coarse.
    \end{enumerate}
\end{proposition}

Finally, let us prove the above three propositions as below.

\begin{proof}[proof of Proposition \ref{proposition:closecomposition}]
    Our goal is to show $\Xi_{\varphi\circ f, \psi\circ g}$ is an entourage on $U$.
    We have
    \begin{align*}
        \Xi_{\varphi\circ f, \psi\circ g}&=(\varphi\circ f)\circ(\psi\circ g)^\dagger\\
        &\subset \varphi\circ (\psi^\dagger \circ \psi)\circ f\circ g^\dagger\circ \psi^\dagger\\
        &=\Xi_{\varphi, \psi}\circ \Ad_{\psi}(\Xi_{f,g}).
    \end{align*}
    By the assumptions  $f\sim g$ and $\varphi\sim \psi$, 
    two sets $\Xi_{f,g}$ and $\Xi_{\varphi, \psi}$ are entourages on $T$ and $U$, respectively.
    Furthermore, $\Ad_{\psi}(\Xi_{f,g})$ is also on $U$ since $\psi$ is controlled.
    Hence the subset 
    $\Xi_{\varphi\circ f, \psi\circ g}\subset \Xi_{\varphi, \psi}\circ \Ad_{\psi}(\Xi_{f,g})$ is also an entourage on $U$.
\end{proof}

\begin{proof}[proof of Proposition \ref{proposition:preserveclose}]
    we may only check that five conditions 
    \eqref{definition:coarsemaps:item:controlled}, 
    \eqref{definition:coarsemaps:item:proper}, 
    \eqref{definition:coarsemaps:item:effectivelyproper}, 
    \eqref{definition:coarsemaps:item:surjective}, 
    and \eqref{definition:coarsemaps:item:coarseeq}
    are preserved under closeness. 
    Firstly, let us check about  \eqref{definition:coarsemaps:item:controlled}.
    Take $E\in \mathcal{E}$.
    Our goal is to show $\Ad_gE\in \mathcal{F}$.
    We have
    \begin{align*}
        \Ad_gE&=g\circ E\circ g^{\dagger}\\
        &\subset g\circ (f^\dagger \circ f) \circ E\circ (f^\dagger\circ f)\circ g ^\dagger\\
        &\subset (f\circ g^\dagger)^\dagger \circ f \circ E\circ f^\dagger\circ (f\circ g ^\dagger)\\
        &=\Xi_{f,g}^\dagger\circ \Ad_{f}E\circ \Xi_{f,g}.
    \end{align*}
    Since $f$ is controlled and $f\sim g$, we obtain $\Xi_{f,g}^\dagger\circ \Ad_{f}E\circ \Xi_{f,g}\in \mathcal{F}$.
    Hence the set $\Ad_gE$ also belongs to $\mathcal{F}$.
    Thus, the map $g$ is controlled.
    By a similar argument, we can prove that condition \eqref{definition:coarsemaps:item:effectivelyproper} is also preserved under closeness.
    It immediately follows from Proposition \ref{proposition:closecomposition} and the above argument that condition \eqref{definition:coarsemaps:item:coarseeq} is preserved under closeness.
    Next, we check about \eqref{definition:coarsemaps:item:proper}.
    Take a bounded set $B\subset T$.
    It suffices to show that $g^\dagger[B]$ is bounded in $S$.
    We have
    \begin{align*}
        g^\dagger[B]&\subset  (f^\dagger\circ f)\circ g^\dagger[B]\\
        &=f^\dagger\circ \Xi_{f,g}[B]
    \end{align*}
    By $\Xi_{f,g}\in \mathcal{F}$, the neighborhood $\Xi_{f,g}[B]$ of the bounded set $B$ is also bounded in $T$.
    Furthermore, since the map $f$ is coarsely proper, $f^\dagger\circ \Xi_{f,g}[B]$ is a bounded set in $S$, and 
    the subset $g^\dagger[B]$ is also.
    Thus, $g$ is also coarsely proper.

    Finally, let us check about \eqref{definition:coarsemaps:item:surjective}.
    Since $f$ is coarsely surjective, there exists $F\in \mathcal{F}$ such that $F[f(S)]=T$.
    Further, the set $\Xi_{f,g}$ belongs to $\mathcal{F}$ by $f\sim g$, and $F\circ \Xi_{f,g}\in \mathcal{F}$. 
    Hence, we obtain
    \[
    F\circ \Xi_{f,g}[g(S)]=F\circ f\circ g^\dagger\circ g[S]
    \supset F\circ f[S]=T.
    \]
    Thus, the equality $F\circ \Xi_{f,g}[g(S)]=T$ holds, and $g$ is coarsely surjective.
\end{proof}

\begin{proof}[proof of Proposition \ref{proposition:products}]
    Assertion \eqref{proposition:products:controlled} follows immediately from the definition of the product coarse structure (cf.~ \cite{DikranjanZava2017}).
    Hence let us show assertion \eqref{proposition:products:coarse}. 
    By assertion \eqref{proposition:products:controlled}, the map
    \[
    g:=(f_\lambda)_{\lambda\in\Lambda\setminus\{\lambda_0\}}
    \colon S\to\prod_{\lambda\in\Lambda\setminus\{\lambda_0\}}T_\lambda
    \]
    is controlled.
    Thus, we may regard
    \[
    (f_\lambda)_{\lambda\in\Lambda}
    \]
    as the product of the coarse map $f_{\lambda_0}$ and the controlled map $g$.
    Hence, it suffices to consider the case $\Lambda=\{0,1\}$, where $f_0$ is coarse and $f_1$ is controlled.
    By assertion \eqref{proposition:products:controlled}, 
    the map $(f_0,f_1)$ is also controlled. 
    Hence we may only check that $(f_0,f_1)$ is coarsely proper.
    Fix a bounded set $B\subset T_0\times T_1$.
    Then there exists bounded sets $B_0\subset T_0$ and $B_1\subset T_1$ such that $B\subset B_0\times B_1$.
    Then, we have
    \[
        (f_0,f_1)^{-1}(B)\subset f_0^{-1}(B_0)\cap f_1^{-1}(B_1)\subset f_0^{-1}(B_0). 
    \]
    Since $f_0$ is coarsely proper, $f_0^{-1}(B_0)$ is bounded in $S$.
    Hence, the subset $(f_0,f_1)^{-1}(B)$ is also bounded in $S$.
    Thus, the map $(f_0,f_1)$ is coarsely proper.
    Therefore, $(f_0,f_1)$ is coarse.
\end{proof}

\subsection{Maps between locally compact Hausdorff groups with the LR-coarse structures}

In this section, we record some basic facts concerning maps between locally compact Hausdorff groups equipped with their LR coarse structures (see Example \ref{example:coarse}). 

\begin{proposition}\label{proposition:grouphom}
    Let $G_1,G_2$ be locally compact Hausdorff groups and $f$ a map from $G_1$ to $G_2$.
    Consider that $G_1$ and $G_2$ are equipped with LR-coarse structure.
    Then the following six assertions hold:
    \begin{enumerate}
        \item\label{proposition:grouphom:item:controlled} Assume that  $f$ is a continuous group-homomorphism. Then $f$ is controlled.
        \item\label{proposition:grouphom:item:controlledanti} Assume that  $f$ is a continuous anti-group-homomorphism. Then $f$ is controlled.
        \item\label{proposition:grouphom:item:coarse} Assume that  $f$ is a topological closed embedding group-homomorphism. Then $f$ is coarse. 
        \item\label{proposition:grouphom:item:anticoarse} Assume that  $f$ is a topological closed embedding anti-group-homomorphism. Then $f$ is coarse. 
        \item\label{proposition:grouphom:item:coarseeq} Assume that  $f$ is an isomorphism of topological groups. Then $f$ is a coarse equivalence. 
        \item\label{proposition:grouphom:item:anticoarseeq} Assume that  $f$ is an anti-isomorphism of topological groups. Then $f$ is a coarse equivalence. 
    \end{enumerate}
    In particular, for a group and its closed subgroup, the inclusion map is coarse.
\end{proposition}

A proof of the above proposition will be given the end of this subsection.

As a corollary of the above proposition, we can obtain the following two properties:

\begin{corollary}\label{corollary:phiwedge}
    For a locally compact Hausdorff group $G$ with LR-coarse structure, the inverse map $g\mapsto g^{-1}$ is a coarse equivalence.
    Furthermore, for $n\in \Z_{\geq 1}$ and $i,j,q\in [n]$ with $i\leq j$, the following holds:
    \begin{enumerate}
        \item $\Phi_{ij}:\rmB(n)\to \rmB(j-i+1), \, (x_{st})_{1\leq s,t,\leq n}\mapsto (x_{st})_{i\leq s,t,\leq j}$ is controlled.
        \item 
        $\bigwedge^q : \GL (n,\R) \to \GL (N,\R), \quad X\mapsto (\det X_{\mathbb{I}_{n,q}(s)\mathbb{I}_{n,q}(t)})_{s,t\in \left[\binom{n}{q}\right]}$ is controlled (see Section \ref{subsection:grouphom} for the notation).
        Here, $N:=\binom{n}{q}$.
    \end{enumerate}
\end{corollary}

\begin{corollary}\label{corollary:log}
    Regard $\mathbb{R}_{>0}$ as a group under multiplication, equip it with the LR-coarse structure and 
    $\R$ with the bounded coarse structure. Then the map
    \[
    \log\colon \mathbb{R}_{>0}\longrightarrow \mathbb{R}
    \]
    is coarse equivalence.
\end{corollary}

Note that the bounded coarse structure on $\R$ coincides with the LR-coarse structure on $\R$.
Since the inverse map $\exp\colon \mathbb{R}\longrightarrow \mathbb{R}_{> 0}$ is also a continuous group homomorphism, the above corollary holds.

Finally,let us show Proposition \ref{proposition:grouphom} as below:

\begin{proof}[proof of Proposition \ref{proposition:grouphom}]
    Firstly, we show assertion \eqref{proposition:grouphom:item:controlled}.
    For a compact subset $C\subset G_1$, 
    $\Ad_f E_C^{\text{LR}}\subset E_{f(C)}^{\text{LR}}$ holds.
    In fact, for each $c_1,c_2\in C$ and $x\in G_1$, 
    we can obtain $f(c_1xc_2^{-1})=f(c_1)f(x)f(c_2)^{-1}$ since $f$ is a group homomorphism.
    By the assumption that $f$ is continuous, $f(C)$ is bounded in $(G_2,\mathcal{E}^\text{LR})$.
    Hence, $f$ is controlled.
    By the similar argument, we can check \eqref{proposition:grouphom:item:controlledanti}.
    
    By assertions \eqref{proposition:grouphom:item:controlled}, \eqref{proposition:grouphom:item:controlledanti} and Proposition \ref{proposition:closedeqqborno}, we can obtain \eqref{proposition:grouphom:item:coarse} and \eqref{proposition:grouphom:item:anticoarse} (see also Remark~\ref{remark:btop}).
    Furthermore, two assertions \eqref{proposition:grouphom:item:coarseeq} and \eqref{proposition:grouphom:item:anticoarseeq} follow from \eqref{proposition:grouphom:item:controlled}, \eqref{proposition:grouphom:item:controlledanti} and the definition of the coarse equivalence.
\end{proof}

\subsection{Asymptotic disjointness}\label{subsection:arandad}

Next, we recall the notion of ``separation'' in the sense of coarse geometry and some properties of this.
Let $(S,\mathcal{E})$ be a coarse space.
we define ``separation'' as below:

\begin{definition}\label{definition:asymptoticallydisjoint}[cf.~\cite{BanakhProtasov2018}]
    Let $A$ and $D$ be subsets of $S$.
    The pair $(A,D)$ is called an \emph{asymptotically disjoint pair} (denoted by $A\pitchfork_\mathcal{E} D$) if any (equivalently, all) of the following equivalent conditions hold:
    \begin{enumerate}
        \item the set $A\cap E[D]$ is bounded in $S$ for any $E\in \mathcal{E}$.
        \item the set $E[A]\cap D$ is bounded in $S$ for any $E\in \mathcal{E}$.
        \item the set $E[A]\cap E[D]$ is bounded in $S$ for any $E\in \mathcal{E}$.
        \item the set $E[A]\cap F[D]$ is bounded in $S$ for any $E,F\in \mathcal{E}$.
    \end{enumerate}
\end{definition}

The following example relates asymptotic disjointness to properness of natural group actions.

\begin{example}\label{example:properad}
    Let $G$ be locally compact Hausdorff group, and $H,L$ be closed subgroups of $G$.
    By Proposition \ref{proposition:properbyside}, the following four conditions are equivalent:
    \begin{enumerate}
        \item The pair $(L,H)$ is asymptotically disjoint in $(G,\mathcal{E}^\text{LR})$ (see Example \ref{example:coarse} for the definition of $\mathcal{E}^\text{LR}$ ).
        \item The natural $L$-action on the homogeneous space $G/H$ is proper.
        \item The natural $H$-action on the homogeneous space $G/L$ is proper.
        \item The diagonal $G$-action on the space $G/H\times G/L$ is proper.
    \end{enumerate}
\end{example}

The following proposition describes how asymptotic disjointness behaves under coarse maps and coarse embeddings.

\begin{proposition}\label{proposition:ad}
    Let $(S,\mathcal{E}),(T,\mathcal{F})$ be coarse spaces and 
    $f$ a map from $S$ to $T$.
    \begin{enumerate}
        \item\label{proposition:ad:item:coarse} If $f$ is coarse, the map $f$ reflects asymptotic disjointness.
        That is, for $A,D\subset S$, if $f(A)\pitchfork_{\mathcal{F}}f(D)$, then $A\pitchfork_\mathcal{E} D$.
        \item\label{proposition:ad:item:emb} If $f$ is a coarse embedding, the map $f$  both preserves and reflects  asymptotic disjointness.
        That is, for $A,D\subset S$, $A\pitchfork_\mathcal{E} D$ if and only if $f(A)\pitchfork_{\mathcal{F}} f(D)$. 
    \end{enumerate}
\end{proposition}

The above proposition immediately follows from Theorem 7.18 in \cite{MiyajiNagayaOgawaOkuda2026}. 
However, for the reader's convenience, we give a proof in the end of this subsection.

As a corollary of the above theorem, we can obtain the following (see also Example \ref{example:properad}):

\begin{corollary}\label{corollary:coarsemapproper}
    Let $G$ be a locally compact Hausdorff group with the LR-coarse structure, $L,H$ closed subgroups of $G$, $(S,\mathcal{E})$ a coarse space, and $\mu:G\to S$ a coarse map.
    If the pair $(\mu(L),\mu(H))$ are asymptotically disjoint, then the natural $L$-action on $G/H$ is proper.
    Furthermore, if $\mu$ is a coarse embedding, the converse claim holds.
\end{corollary}

Finally, we shall check Proposition \ref{proposition:ad}.

\begin{proof}[proof of Proposition \ref{proposition:ad}]
    We first show Assertion \eqref{proposition:ad:item:coarse}.
    Fix two subsets $A,D\subset S$, and assume $f(A)\pitchfork_\mathcal{F} f(D)$.
    Our goal is to show $A\pitchfork_\mathcal{E} D$. 
    Take $E\in \mathcal{E}$.
    Then, we have:
    \begin{align*}        
        A\cap E[D] &\subset f^{\dagger}\circ f [A\cap E[D]] \\
        &\subset f^{\dagger}[f(A)\cap f(E[D])]\\
        &\subset f^{\dagger}[f(A)\cap f\circ E\circ f^{\dagger}\circ f[D])]\\
        &= f^{\dagger}[f(A)\cap \Ad_f(E)[f(D)]].
    \end{align*}
    Since $f$ is controlled, $\Ad_f(E)$ belongs to $\mathcal{F}$.
    By the assumption $f(A)\pitchfork_\mathcal{F} f(D)$, 
    the intersection
    $f(A)\cap \Ad_f(E)[f(D)]$ is bounded in $T$.
    Hence $f^{\dagger}[f(A)\cap \Ad_f(E)[f(D)]]$ is bounded in $S$ since $f$ is coarse.
    Thus, the set $A\cap E[D]$ is also bounded in $S$, and we obtain $A\pitchfork_\mathcal{E} D$.

    Next, let us show Assertion \eqref{proposition:ad:item:emb}.
    Fix an asymptotically disjoint pair $(A,D)$ in $S$.
    It suffices to show $f(A)\pitchfork_\mathcal{F} f(D)$.
    Take $F\in \mathcal{F}$.
    Then, the following inclusions hold:
    \begin{align*}
        f(A)\cap F[f(D)]&= f\circ f^\dagger[f(A)\cap F[f(D)]] \quad (\text{since } f(A)\cap F[f(D)] \text{ is a subset of }\text{Im}f)\\
        &\subset f[f^\dagger\circ f[A]\cap   f^\dagger \circ F\circ f[D]]\\
        &\subset f[\Ad_{f^\dagger} (\diag(T))[A]\cap \Ad_{f^\dagger}(F)[D]].
    \end{align*}
    Since $f$ is effectively proper, two sets $\Ad_{f^\dagger} (\diag(T))$ and $\Ad_{f^\dagger}(F)$ belong in $\mathcal{E}$.
    By the assumption $A\pitchfork_\mathcal{E} D$,
    The intersection $\Ad_{f^\dagger} (\diag(T))[A]\cap \Ad_{f^\dagger}(F)[D]$ is bounded in $S$.
    Hence the image $f[\Ad_{f^\dagger} (\diag(T))[A]\cap \Ad_{f^\dagger}(F)[D]]$ is bounded in $T$ since $f$ is controlled.
    Thus, the set $f(A)\cap F[f(D)]$ is also bounded in $T$.
\end{proof}

\section{Proper Actions on the homogeneous space $\rmB(n)/H$}\label{section:properaction}

In this section, we establish the main results of this paper concerning criteria for proper actions in the point of view of coarse geometry.

\subsection{Coarse Maps and Proper Actions on the homogeneous space $\rmB(n)/H$}\label{subsection:coarsemapproper}

In this section, we construct several coarse maps from $\rmB(n)$ and derive sufficient conditions for properness of appropriate actions from their coarse properness.

Throughout this paper, we fix the following notation:

\begin{itemize}
    \item To simplify notation, for any $n\in\mathbb{N}$ and any subset $S\subset \text{M}(n,\R)$, we use the following same notation:
    \[
    \|-\|:S\to\R_{\geq 0},\quad  X:=(x_{ij})_{1\leq i,j\leq n}\mapsto \|X\|:=\max\{|x_{ij}|\mid 1\leq i,j\leq n\}.
    \]
    \item For $n\in \Z_{\geq 1}$, 
    we define the following maps:
    \begin{align*}
        \Lambda_0(n)&:=\{(i,j,q,\varepsilon)\in [n]^3\times \{\pm 1\}\mid q=\varepsilon=1\text{ and } \left(i=j \text{ or }(i,j)=(1,n)\right)\}, \\
        \Lambda_1(n)&:=\{(i,j,q,\varepsilon)\in [n]^3\times \{\pm 1\}\mid i\leq j,1\leq q \leq j-i+1\}.
    \end{align*}
    Further, for a nonempty subset $\Lambda\subset \Lambda_1(n)$, put \[
        \mu_\Lambda : \rmB(n)\to \R^{ \Lambda}, \quad X:=(x_{ij})_{i,j\in [n]}\mapsto (\log \|\bigwedge^q\Phi_{ij}(X^\varepsilon)\|)_{(i,j,q,\varepsilon)\in \Lambda}.
        \]
    In particular, let us denote $\mu_{\Lambda_0}$ by $\mu_0$, and $\mu_{\{\lambda\}}$ by $\mu_\lambda$ for each $\lambda\in \Lambda_1(n)$.
\end{itemize}

Then, we introduce the following two Theorems:

\begin{theorem}\label{theorem:mu0}
    Fix $n\in \Z_{\geq 1}$.
    Then the following holds:
    \begin{enumerate}
        \item\label{theorem:mu0:item:controlled} For each $\lambda\in \Lambda_1(n)$, the map $\mu_{\lambda}$ is controlled.
        Here, $\rmB(n)$ is equipped with the LR-coarse structure and $\R^{\{\lambda\}}\cong \R$ is equipped with the bounded coarse structure (see Example \ref{example:coarse}).
        \item\label{theorem:mu0:item:coarse} The map $\mu_0$ is coarse.
    \end{enumerate}
\end{theorem}

\begin{theorem}\label{theorem:maintheoremimage}
    Fix $n\in \Z_{\geq 1}$ and a set $\Lambda$ with $\Lambda_0(n)\subset \Lambda\subset \Lambda_1(n)$. 
    Then $\mu_\Lambda : \rmB(n)\to \R^{ \Lambda}$ is coarse.
    Here, $\rmB(n)$ is equipped with the LR-coarse structure and $\R^{ \Lambda}$ is equipped with the bounded coarse structure.
    In particular, for closed subgroups $L,H$ of $\rmB(n)$, 
    if the pair $(\mu_\Lambda(L), \mu_\Lambda(H))$ is asymptotically disjoint in $(\R^{ \Lambda},\mathcal{E}_d)$, 
    the natural $L$-action on $\rmB(n)/H$ is proper (see Corollary \ref{corollary:coarsemapproper}). 
\end{theorem}

To prove the above two theorems, we prepare the following two lemmas:

\begin{lemma}\label{lemma:submultiplicative} 
    Let $G$ be a locally compact Hausdorff group. 
    Equip both $G$ and $\R_{>0}$ with their compact bornologies (see Example \ref{example:bornology}) and their respective LR-coarse structures (see Example \ref{example:coarse}). 
    Let $f:G\to \R_{>0}$ be a bornological map. 
    Assume that there exists $a\geq 1$ such that \[ f(xy)\leq a f(x)f(y) \] for all $x,y\in G$. Then $f$ is controlled. 
    In particular, every continuous map $f:G\to\R_{>0}$ satisfying the above inequality is controlled. 
\end{lemma}

\begin{lemma}\label{lemma:logsubmultiplicative}
    Let $G$ be a locally compact Hausdorff group and $f$ a bornological map from $G$ to $\R_{> 0}$.
    Assume that there exists $a\geq 1$ such that, for any $x,y\in G$, $f(xy)\leq af(x)f(y)$.
    Then, the map $\log\circ f : G\to \R$ is controlled.
    Here, $G$ is equipped with the LR-coarse structure and $\R$ is equipped with the bounded coarse structure.
\end{lemma}

Lemma \ref{lemma:logsubmultiplicative} immediately follows from Lemma \ref{lemma:submultiplicative} and Corollary \ref{corollary:log}.
Hence we only prove  Lemma \ref{lemma:submultiplicative} as below:

\begin{proof}[proof of Lemma \ref{lemma:submultiplicative}]
    Take $a\geq 1$ such that, for any $x,y\in G$, $f(xy)\leq af(x)f(y)$.
    We fix a compact subset $C\in \mathcal{K}(G)$ (The notation $\mathcal{K}(G)$ is fixed in Example \ref{example:coarse}).
    There exists $R\geq 1$ such that 
    \[
    f(C)=f(C^{-1})\subset \left[ \frac{1}{R},R \right].
    \]
    Put $K:=[\frac{1}{a^2R^2},a^2R^2]$
    Take a pair $(x,y)\in E_C^\text{LR}$ and choose $c_1,c_2\in C$ with $x=c_1yc_2^{-1}$.  
    Then we have
    \begin{align*}
    \frac{1}{a^2R^2}\leq \frac{f(x)}{a^2f(c_1^{-1})f(x)f(c_2)}\leq \frac{f(x)}{f(c_1^{-1}xc_2)}&=\frac{f(x)}{f(y)}\\
    &=\frac{f(c_1yc_2^{-1})}{f(y)}\leq \frac{a^2f(c_1)f(y)f(c_2)}{f(y)}\leq a^2R^2.
    \end{align*}
    Hence we can obtain $\frac{1}{a^2R^2}f(y)\leq f(x)\leq a^2R^2f(y)$.
    Thus, $\Ad_fE_C^\text{LR}$ is a subset of $E_{K}^{\text{L}}\in  \mathcal{E}^\text{L}$.
    Since $\R_{>0}$ is abelian, $\mathcal{E}^\text{L}$ coincides with $\mathcal{E}^\text{LR}$.
    Therefore, the map $f: G\to \R_{>0}$ is controlled.
\end{proof}

Before proving two Theorems \ref{theorem:mu0} and \ref{theorem:maintheoremimage}, we illustrate how Lemma \ref{lemma:logsubmultiplicative} applies in concrete examples as below:

\begin{example}\label{example:submultiplicative}
    By Lemma \ref{lemma:logsubmultiplicative}, we can obtain the following controlled map.
    \begin{itemize}
        \item For a locally compact Hausdorff group $G$ and a continuous submultiplicative function $f:G\to \R_{>0}$, the map $\log\circ f$ is controlled.
        In particular, consider the submultiplicative norm on the $n\times n$-matrix algebra $\text{M}(n,\R)$, and fix a  closed multiplicative subgroup $G$ of $\GL(n,\mathbb{R})$.
        If the restriction of submultiplicative norm on $G$ is valued on $\R_{>0}$, then the composition of $\log$ and the norm is a controlled map from $G$ to $\R$.
        \item Fix $n\in \Z_{\geq 1}$, and let $G$ be a multiplicative closed subgroup of $\GL(n,\R)$.
         Then the composition $\log$ and the map $\|-\|$ from $G$ to $\R_{>0}$ (defined in Subsection \ref{subsection:coarsemapproper} ) is a controlled map from $G$ to $\R$.
        In fact, for $X,Y\in G$, we have:
        \[
        \|XY\|
        =\max\left\{\left| \sum_{1\leq k\leq n}x_{ik}y_{kj} \right| \middle| 1\leq i,j\leq n \right\}
        \leq \sum_{1\leq k\leq n}\|X\|\|Y\|=n\|X\|\|Y\|.
        \]
    \end{itemize}
\end{example}

Let us prove two Theorems \ref{theorem:mu0} and \ref{theorem:maintheoremimage}.
Theorem \ref{theorem:maintheoremimage} immediately follows from assertion \eqref{proposition:products:coarse} in Proposition \ref{proposition:products} and Theorem \ref{theorem:mu0}.
Thus, we may only show Theorem \ref{theorem:mu0} as below:

\begin{proof}[proof of Theorem \ref{theorem:mu0}]
    By Proposition~\ref{proposition:composition}, Corollary~\ref{corollary:phiwedge} and Lemma~\ref{lemma:logsubmultiplicative}, assertion \eqref{theorem:mu0:item:controlled} immediately holds.
    Furthermore,  the map $\mu_0$ is also controlled by 
    Proposition \ref{proposition:products}.

    Next, we shall show that $\mu_0$ is coarsely proper.
    Take a compact subset $C\subset \R^{\Lambda_0(n)}$.
    Then, there exists $R\geq 1$ such that $C\subset [-R,R]^{\Lambda_0(n)}$.
    Further, the following inclusion holds:
    \begin{align*}
        \mu_0^{-1}(C)\subset \left\{
        \begin{pmatrix}
        x_{11} & x_{12} & \cdots & x_{1n}\\
        0      & x_{22} & \cdots & x_{2n}\\
        \vdots & \ddots & \ddots & \vdots\\
        0      & \cdots & 0      & x_{nn}
        \end{pmatrix}
        \in \rmB(n)
        \ \middle|\
        \substack{
        e^{-R}\leq |x_{ii}|\leq e^R
        \quad (\forall i\in[n]),\\
        |x_{st}|\leq e^R
        \quad (\forall\,1\leq s<t\leq n)
        }
        \right\}.
    \end{align*}
    The right side of the above is compact in $\rmB(n)$.
    Hence, the set $\mu_0^{-1}(C)$ is bounded in $(\rmB(n),\mathcal{E}^\text{LR})$.
    Thus, the map $\mu_0$ is coarsely proper.
\end{proof}

\subsection{Controlled maps and Asymptotic Disjointness}

In this section, we characterize asymptotic disjointness in terms of controlled maps and apply this to give a sufficient condition for properness of the $L$-action on $\R^{n-1}$,  where $L$ is a closed subgroup of $\rmB(n)$ (see Theorem \ref{theorem:maintheoremrestrictionl}).

Firstly, we give a coarse-geometric characterization of asymptotic disjointness as below:

\begin{theorem}\label{theorem:adfunc}
    Let $(S,\mathcal{E})$ be a coarsely connected coarse space, and $A,D$ subsets of $S$.
    Then the following two conditions are equivalent: 
    \begin{enumerate}
        \item\label{theorem:adfunc:item:ad} $A\pitchfork_\mathcal{E} D$. That is, the pair $(A,D)$ is asymptotically disjoint in $(S,\mathcal{E})$ (see Definition \ref{definition:asymptoticallydisjoint}).
        \item\label{theorem:adfunc:item:func} There exists a coarse space $(T,\mathcal{F})$ and a controlled map $f:S\to T$ such that $f(D)$ is bounded in $T$ and the restriction $f|_A$ is coarsely proper.
        Here, $(A,\mathcal{E}|_A)$ is subcoarse space of $(S,\mathcal{E})$.
    \end{enumerate}
\end{theorem}

\begin{proof}
    Firstly, assume \eqref{theorem:adfunc:item:ad}, and prove \eqref{theorem:adfunc:item:func}.
    Put $E_D:=(D\times D)\cup \diag(S)$.
    We shall begin by proving that the following family forms a coarse base on $S$:
    \[
        \widetilde{\mathcal{E}_0}:=\{E\circ E_D\circ F\mid E,F\in \mathcal{E}\}.
    \]
    it suffices to check the following four conditions:    
    \begin{enumerate}
        \item[(i)'] There exists $E_0 \in \widetilde{\mathcal{E}_0}$ such that  $\diag (S) \subset E_0$.
        \item[(ii)'] For each $E \in \widetilde{\mathcal{E}_0}$, there exists $E_0 \in \widetilde{\mathcal{E}_0}$ such that $E^\dagger \subset E_0$.
        \item[(iii)'] For each $E_1,E_2 \in \widetilde{\mathcal{E}_0}$, 
        there exists $E_0 \in \widetilde{\mathcal{E}_0}$ such that $E_1 \cup E_2 \subset E_0$.
        \item[(iv)'] For each $E_1,E_2 \in \widetilde{\mathcal{E}_0}$, 
        there exists $E_0 \in \widetilde{\mathcal{E}_0}$ such that $E_1 \circ E_2 \subset E_0$.
    \end{enumerate}
    Assertion (i)'-- (iii)' immediately follows from:
    \begin{itemize}
        \item $\diag(S)\subset \diag(S)\circ E_D\circ \diag(S)$,
        \item $(E\circ E_D\circ F)^\dagger=F^\dagger \circ E_D\circ E^\dagger$ for $E,F\in \mathcal{E}$,
        \item $(E_1\circ E_D\circ F_1)\cup (E_2\circ E_D\circ F_2)\subset (E_1\cup E_2)\circ E_D\circ (F_1\cup F_2)$ for $E_1,E_2,F_1,F_2\in \mathcal{E}$.
    \end{itemize}
    Let us check (iv)'.
    For $E_1,E_2,F_1,F_2\in \mathcal{E}$, the following holds:
    \begin{align*}
        &(E_1\circ E_D\circ F_1)\circ (E_2\circ E_D\circ F_2)\\
        &=(E_1\circ ((D\times D)\cup \diag(S))\circ F_1)\circ (E_2\circ ((D\times D)\cup \diag(S))\circ F_2)\\
        &=((E_1\circ (D\times D)\circ F_1)\circ (E_2\circ (D\times D)\circ F_2))
        \cup((E_1\circ (D\times D)\circ F_1)\circ (E_2\circ  \diag(S)\circ F_2))\\
        &\quad \cup((E_1\circ \diag(S)\circ F_1)\circ (E_2\circ (D\times D)\circ F_2))
        \cup ((E_1\circ \diag(S)\circ F_1)\circ (E_2\circ \diag(S)\circ F_2))\\
        &\subset(E_1\circ (D\times D)\circ F_2)
        \cup(E_1\circ (D\times D)\circ (F_1\circ E_2\circ F_2))\\
        &\quad\cup((E_1\circ F_1\circ E_2)\circ (D\times D)\circ F_2)
        \cup((E_1\circ F_1\circ E_2)\circ \diag(S)\circ F_2)\\
        &\subset(E_1\circ E_D\circ F_2)
        \cup(E_1\circ E_D\circ (F_1\circ E_2\circ F_2))
        \cup((E_1\circ F_1\circ E_2)\circ E_D\circ F_2).
    \end{align*}
    Since $\mathcal{E}$ is closed under the composition and assertion (iii)', there exists $\widetilde{E}\in \widetilde{\mathcal{E}_0}$ such that $(E_1\circ E_D\circ F_2)
    \cup(E_1\circ E_D\circ (F_1\circ E_2\circ F_2))
    \cup((E_1\circ F_1\circ E_2)\circ E_D\circ F_2)\subset \widetilde{E}$.
    Hence, $(E_1\circ E_D\circ F_1)\circ (E_2\circ E_D\circ F_2)$ is also a subset of $\widetilde{E}\in \widetilde{\mathcal{E}_0}$.
    Thus, assertion (iv)' holds.
    Therefore, the family $\widetilde{\mathcal{E}_0}$ forms a coarse base.
    Next, consider the identity map $\id$ from $(S,\mathcal{E})$ to $(S,\langle \widetilde{\mathcal{E}_0}\rangle )$.
    It immediately holds that this map is controlled.
    Hence, our goal is to show that  the image $\id(D)=D$ is bounded in $(S,\langle \widetilde{\mathcal{E}_0}\rangle )$ and the restriction $\id|_A$ from $(A,\mathcal{E}|_A)$ to $(S,\langle \widetilde{\mathcal{E}_0}\rangle )$ is coarsely proper.
    By $D\times D\subset \diag(S)\circ E_D\circ \diag(S)$, the set $D$ is bounded in $(S,\langle \widetilde{\mathcal{E}_0}\rangle )$.
    Take a bounded set $B$ in $(S,\langle \widetilde{\mathcal{E}_0}\rangle )$.
    We may assume that $B$ is nonempty.
    Then, there exists a point $p\in S$ and entourages $E,F\in \mathcal{E}$ such that $B\subset E\circ E_D\circ F[p]$.
    \begin{align*}
        (\id|_A)^{-1}(B)&=A\cap B\\
        &\subset A\cap E\circ E_D\circ F[p]\\
        &\subset A\cap (E[D]\cup E\circ F[p])\\
        &\subset (A\cap E[D])\cup (A\cap E\circ F[p])
    \end{align*}
    By the assumption \eqref{theorem:adfunc:item:ad}, 
    the set $A\cap E[D]$ is bounded in $(S,\mathcal{E})$, and it is also bounded in $(A,\mathcal{E}|_A)$.
    Furthermore, the intersection $(A\cap E\circ F[p])$ is a bounded set in $(A,\mathcal{E}|_A)$.
    Since $(S,\mathcal{E})$ is coarsely connected, 
    $(A,\mathcal{E}|_A)$ is also coarsely connected.
    Hence, the union $(A\cap E[D])\cup (A\cap E\circ F[p])$ is bounded in $(A,\mathcal{E}|_A)$, and 
    the subset $(\id|_A)^{-1}(B)\subset (A\cap E[D])\cup (A\cap E\circ F[p])$ is also.
    Thus, the restriction $\id|_A$ from $(A,\mathcal{E}|_A)$ to $(S,\langle \widetilde{\mathcal{E}_0}\rangle )$ is coarsely proper.
    Therefore, condition \eqref{theorem:adfunc:item:func} holds.
    
    Next, let us prove the converse claim.
    Let  $(T,\mathcal{F})$ be a coarse space and $f$ a controlled map from $S$ to $T$.
    Suppose that $f(D)$ is bounded in $T$ and the restriction $f\circ\iota_A$ from $(A,\mathcal{E}|_A)$ to $(T,\mathcal{F})$ is coarsely proper.
    Take an entourage $E\in \mathcal{E}$.
    Our goal is to show that $A\cap E[D]$ is bounded in $(S,\mathcal{E})$.
    \begin{align*}
        A\cap E[D]&\subset \iota_A\circ \iota_A^\dagger \circ E[D] \\
        &\subset \iota_A\circ \iota_A^\dagger \circ  f^{\dagger}\circ f\circ E\circ f^\dagger\circ f\circ [D]\\
        &= \iota_A\circ (f\circ \iota_A)^{\dagger}\circ \Ad_f(E)\circ f[D]
    \end{align*}
    By controlledness of $f$, $\Ad_f(E)$ belongs to $\mathcal{F}$.
    Since $f(D)$ is bounded in $T$ and $f\circ \iota_A$ is coarsely proper, the set $\Ad_f(E) [f(D)]$ is bounded in $T$ and $(f\circ \iota_A)^{\dagger}\circ \Ad_f(E) [f(D)]$ is also bounded in $(A,\mathcal{E}|_A)$.
    Hence, the subset $A\cap E[D]\subset \iota_A\circ (f\circ \iota_A)^{\dagger}\circ \Ad_f(E)\circ f[D]$ is bounded in $(S,\mathcal{E})$.
    Thus, we obtain $A\pitchfork_\mathcal{E} D$.
\end{proof}

Next, we introduce a sufficient condition of properness of appropriate actions by applying Theorem \ref{theorem:adfunc} as below.
To simplify notation, we adopt the convention that:
\[
    \bigwedge^0 X := 1\in \GL(1,\R).
\]

For $n\in \Z_{\geq 1}$, fix the following notation:
\begin{align*}
    \Theta(n):=\{(i,q,\varepsilon)\in [n]\times [n]\times \{\pm 1\}\mid 1\leq q\leq n-i\},
\end{align*}

and for a nonempty subset $\Theta\subset \Theta(n)$, put 
\begin{align*}
    \varphi_\Theta:\rmB(n)\to \R^{ \Theta}, \, X\mapsto \left(\log\left(
    \frac{\|\bigwedge^q \Phi_{in}(X^\varepsilon)\|}{\max\{\|\bigwedge^q\Phi_{in-1}(X^\varepsilon)\|,\|\bigwedge^{q-1}\Phi_{in-1}(X^\varepsilon)\||x_{nn}^\varepsilon|\}}
    \right)\right)_{(i,q,\varepsilon)\in \Theta}.
\end{align*}
for every invertible matrix $X$, so that $\|\bigwedge^0 X\|=1$. 
For simplicity, we denote $\varphi_\theta:=\varphi_{\{\theta\}}$ for each $\theta \in \Theta(n)$.

Next, we focus on the following action.
For $n\in \Z_{\geq 2}$,
let $\rho$ be the $\rmB(n)$-action on $\R^{n-1}$ as below:
\begin{align*}
    \rmB(n)\times \R^{n-1}&\to \R^{n-1}\\
    \left(
    \begin{pmatrix}
        \Phi_{1n-1}(X) &v \\
        0 & \Phi_{nn}(X)
    \end{pmatrix},w
    \right)
    &\mapsto  \frac{\Phi_{1n-1}(X)w+v}{\Phi_{nn}(X)}.
\end{align*} 
Here, $v:=(x_{1n},\cdots,x_{n-1n})^T\in \R^{n-1}$.

In the above setting, it is worth emphasizing that $\R^{n-1}$  can be regarded as the homogeneous space $\rmB(n)/H$.
Here, 
\[
H:=\left\{X\in \rmB(n)\mid x_{in}=0 \text{ for all } 1\leq i\leq n-1 \right\}.
\]

Applying Theorem \ref{theorem:adfunc} to $\rmB(n)$, we obtain the following theorem:

\begin{theorem}\label{theorem:maintheoremrestrictionl}
    Fix $n\in \Z_{\geq 2}$ and a non empty subset $\Theta\subset \Theta(n)$. 
    Let $L$ be a closed subgroup of $\rmB(n)$.
    Assume that the restriction $\varphi_{\Theta}|_L$ is coarsely proper.
    Then the $L$-action $\rho|_L$ on $\R^{n-1}$ is proper.
\end{theorem}

\begin{proof}
    Hence, by Theorem \ref{theorem:adfunc}, it suffice to show that, the following two assertions:
    \begin{enumerate}
        \item\label{proof:item:controlled} the map $\varphi_\Theta$ is controlled.
        \item\label{proof:item:bd} the image $\varphi_\Theta(H)$ is bounded in $\R^{\Theta}$.
    \end{enumerate}
    Firstly, let us show assertion \eqref{proof:item:controlled}.
    Take $\theta:=(i,q,\varepsilon)\in \Theta$.
    Then one can see that the following equality holds: 
    \[
    \varphi_{\theta}=\mu_{(i,n,q,\varepsilon)}-\max\{\mu_{(i,n-1,q,\varepsilon)},\mu_{(i,n-1,q-1,\varepsilon)}+ \mu_{(n,n,1,\varepsilon)}\}.
    \]
    Here, in case $q-1=0$, $\mu_{(i,n-1,0,\varepsilon)}$ is  defined by the constant map with value $0$, in particular, $\mu_{(i,n-1,0,\varepsilon)}$ is controlled.
    Since the four functions $\mu_{(i,n,q,\varepsilon)}, \mu_{(i,n-1,q,\varepsilon)}, \mu_{(i,n-1,q-1,\varepsilon)}$ and $\mu_{(n,n,1,\varepsilon)}$ are all controlled by Theorem \ref{theorem:mu0}, the map $\varphi_{\theta}$ is also controlled (see also Example \ref{example:controlledtoR}).
    Thus, the map $\varphi_{\Theta}$ is also controlled by Proposition \ref{proposition:products}.   
    Next, we check assertion \eqref{proof:item:bd}.
    For all $\theta:=(i,q,\varepsilon)\in \Theta$, $\varphi_\theta(H)=\{0\}$.
    In fact, for any $X\in H$, we have
    \begin{align*}
        \left\|\bigwedge^q \Phi_{in}(X^\varepsilon)\right\|
        &=\max \left\{ \bigl| \det ((\Phi_{in}(X^\varepsilon))_{IJ} ) \bigr| \middle| I,J\in \binom{[n-i+1]}{q} \right\} \\
        &=\max \left\{ \left|\det \left( 
        \begin{pmatrix}
            \Phi_{in-1}(X^\varepsilon) & 0\\
            0 & x_{nn}^\varepsilon
        \end{pmatrix}
        _{IJ} \right) \right| \middle| I,J\in \binom{[n-i+1]}{q} \right\} \\
        &=\max\left\{\left\|\bigwedge^q\Phi_{in-1}(X^\varepsilon)\right\|,\left\|\bigwedge^{q-1}\Phi_{in-1}(X^\varepsilon)\right\||x_{nn}^\varepsilon|\right\}. 
    \end{align*}
    Hence, the equality $\varphi_\Theta(H)=\{0\}\subset \R^{\Theta}$ holds.
    Thus, the image $\varphi_\Theta(H)$ is bounded in $\R^{\Theta}$.
\end{proof}

\section{Examples of Applications of the Main Theorem}

In this section, we give two examples of applications of two Theorems \ref{theorem:maintheoremimage} and \ref{theorem:maintheoremrestrictionl}, and develop some auxiliary results used in their applications.

\subsection{Examples of Applications of Theorem \ref{theorem:maintheoremimage}}

In this subsection, we illustrate the theorem with a test case, showing that Theorem \ref{theorem:maintheoremimage} recovers the following proper action.

\begin{example}
    Let us consider the following two closed subgroups of $\rmB (4)$: 
  \begin{align*}
    L&:=\left\{
      \begin{pmatrix}
        1 & t & s+\frac{t^2}{2} & u  \\
        0 & 1 & t & s+\frac{t^2}{2}  \\
        0 & 0 & 1 & t  \\
        0 & 0 & 0 & 1  \\
      \end{pmatrix} \middle| 
      s,t,u\in \R
      \right\}, \\
      H&:=\left\{\begin{pmatrix}
        x_{11} & x_{12} & 0 & 0  \\
        0 & x_{22} & 0 & 0  \\
        0 & 0 & x_{33} & x_{34}  \\
        0 & 0 & 0 & x_{44}  \\
      \end{pmatrix} \middle| 
      x_{ii}\in \R^\times \, (i=1,2,3,4), x_{jj+1} \in \R \, (j=1,3)
      \right\}.
  \end{align*}
  Then one can see that the $L$-action on $\rmB(4)/H$ is proper.
  Let us check the properness of the action by Theorem \ref{theorem:maintheoremimage}.
  Consider the set:
  \[
  \Lambda:=\{(i,j,1,1)\mid (i,j)=(1,1), (2,2),(3,3),(4,4),(1,4),(1,2),(2,3),(3,4)\},
  \]
  and fix $R\geq 1$.
  For simplicity, let us put $\mu_{(i,j)}:=\mu_{(i,j,1,1)}$ for each $(i,j,1,1)\in \Lambda$.
  We shall prove that the intersection $E_R[\mu_{\Lambda}(L)]\cap \mu_{\Lambda}(H)$ is bounded in $\R^\Lambda \cong \R^8$. Put a bounded set:
  \[
   B:=\left\{\begin{pmatrix}
        x_{11} & x_{12} & 0 & 0  \\
        0 & x_{22} & 0 & 0  \\
        0 & 0 & x_{33} & x_{34}  \\
        0 & 0 & 0 & x_{44}  \\
      \end{pmatrix} \middle| 
      |x_{ii}|\in [e^{-R},e^R] \, (i=1,2,3,4),\, -e^{3R}\leq x_{jj+1}\leq e^{3R} \, (j=1,3)
      \right\}.
  \]
  It suffices to show that the following inclusion holds:
  \[
  E_R[\mu_{\Lambda}(L)]\cap \mu_{\Lambda}(H)\subset \mu_{\Lambda}(B).
  \]
  Take a point $p\in E_R[\mu_{\Lambda}(L)]\cap \mu_{\Lambda}(H)$.
  Then there exists $h\in H, \ell (s,t,u)\in L$ such that $p=\mu_\Lambda(h)$ and $d\bigl(\mu_\Lambda(h),\mu_\Lambda(\ell(s,t,u))\bigr)\leq R$.
  For each $i\in [4]$, the inequality $e^{-R}\leq |x_{ii}|\leq e^R$ holds by $|\mu_{(i,i)}(h)-\mu_{(i,i)}(\ell (s,t,u))|\leq R$.
  Then, we have:
  \begin{align*}
      |\mu_{(2,3)}(h)-\mu_{(2,3)}(\ell (s,t,u))|\leq R, \\
      |\log(\max\{|x_{22}|,|x_{33}|\})-\log(\max\{1,|t|\})|\leq R, \\
      e^{-R}\leq \frac{\max\{1,|t|\}}{\max\{|x_{22}|,|x_{33}|\}}\leq e^R, \\
      e^{-2R}\leq \max\{1,|t|\}\leq e^{2R}. 
  \end{align*}
  Hence, it leads to the following:
  \begin{align*}
      |\mu_{(1,2)}(h)-\mu_{(1,2)}(\ell (s,t,u))|\leq R, \\
      |\log(\max\{|x_{11}|,|x_{22}|,|x_{12}|\})-\log(\max\{1,|t|\})|\leq R,\\ 
      e^{-R}\leq \frac{\max\{|x_{11}|,|x_{22}|,|x_{12}|\}}{\max\{1,|t|\}}\leq e^R,\\ 
      e^{-3R}\leq \max\{|x_{11}|,|x_{22}|,|x_{12}|\}\leq e^{3R},\\ 
      |x_{12}|\leq e^{3R}.\\ 
  \end{align*}
  By a similar argument, we can also obtain $|x_{34}|\leq e^{3R}$.
  Hence the inclusion $E_R[\mu_{\Lambda}(L)]\cap \mu_{\Lambda}(H)\subset \mu_{\Lambda}(B)$ holds, and $E_R[\mu_{\Lambda}(L)]\cap \mu_{\Lambda}(H)$ is bounded in $\R^\Lambda\cong \R^8$.
  Thus, the pair $(\mu_{\Lambda}(L),\mu_{\Lambda}(H))$ is asymptotically disjoint.
  Therefore, by Theorem \ref{theorem:maintheoremimage}, the natural $L$-action on $\rmB(4)/H$ is proper.
\end{example}

\subsection{Closeness of functions valued on $\R_{>0}$}\label{subsection:closeness}
In this section, we prepare several propositions to facilitate the applications of Theorem \ref{theorem:maintheoremrestrictionl} discussed in the next section.

\begin{proposition}\label{proposition:maxplus}
    Let $S$ be a set, $\{f_i|i=0,1,2\}$ a family of functions from $S$ to $\R_{\geq 0}$ such that $f_0$ values on $\R_{>0}$, and $\{a_i\mid i=0,1,2\}$ a family of elements in $\R_{>0}$.
    Then the following four functions are pairwise close:
    \[
    \max_{i=0,1,2}f_i,\,\max\{a_0f_0+a_1f_1,a_2f_2\},\, \max\{a_0f_0,a_1f_1+a_2f_2\},\, \sum_{i=0}^2a_if_i.
    \]
    Here, $\R_{>0}$ is equipped with the LR-coarse structure.
\end{proposition}

the above proposition immediately follows from the following inequalities:
\begin{align*}
    \left(\min_{i=0,1,2}{a_i}\right)\left( \max_{i=0,1,2}f_i \right)\leq \max\{a_0f_0+a_1f_1,a_2f_2\}\leq \sum_{i=0}^2a_if_i\leq 3\left(\max_{i=0,1,2}{a_i}\right) \left( \max_{i=0,1,2}f_i \right),\\
     \left(\min_{i=0,1,2}{a_i}\right) \left( \max_{i=0,1,2}f_i \right)\leq \max\{a_0f_0,a_1f_1+a_2f_2\}\leq \sum_{i=0}^2a_if_i\leq 3\left(\max_{i=0,1,2}{a_i}\right)\left( \max_{i=0,1,2}f_i \right).
\end{align*}

Furthermore, the following proposition also holds:

\begin{proposition}\label{proposition:maxreduction}
    Let $S$ be a set, $\{f_i|i=0,1,2\}$ a family of functions from $S$ to $\R_{\geq 0}$ such that $f_0$ values on $\R_{>0}$.
    Assume that there exists $C\geq 1$ such that $f_2\leq C\max\{f_0,f_1\}$.
    Then the maps $\max\{f_0,f_1,f_2\}$ and $\max\{f_0,f_1\}$ are close.
    Here, $\R_{>0}$ is equipped with the LR-coarse structure.
\end{proposition}

The above proposition immediately holds by the following:
\[
    \max\{f_0,f_1,f_2\}\leq C\max\{f_0,f_1\}\leq C\max\{f_0,f_1,f_2\}.
\]

Since the closeness is a equivalence relation, Propositions \ref{proposition:maxplus} and \ref{proposition:maxreduction} imply the following corollary:

\begin{corollary}\label{corollary:maxfull}
    Let $S$ be a nonempty set, $I$ a finite set with $\# I\geq 2$, $J_1,J_2$ subsets of $I$, 
    $\{f_i\mid i\in I\}$ a family of functions from $S$ to $\R_{\geq 0}$,
    $\{a_{(\sigma,i)} \mid \sigma =1,2, \, i\in I\setminus J_\sigma\}$ a family of elements of $\R_{>0}$, and $\mathcal{P},\mathcal{Q}$ partitions of $I$.
    Assume that there exists $i_0\in I\setminus (J_1\cup J_2)$ such that $f_{i_0}$ values on $\R_{>0}$, and for each $\sigma\in \{1,2\}$ and $j\in J_\sigma$, there exists $C_j\geq 1$ such that $f_j\leq C_j\max\{f_i\mid i\in I\setminus J_\sigma \}$.
    Then the following two functions are pairwise close:
    \[
    \max\left\{\sum_{i\in P\setminus J_1}a_{(1,i)}f_i\middle| P\in \mathcal{P}\text{ with }P\neq J_1\right\},\quad 
    \max\left\{\sum_{i\in Q\setminus J_2}a_{(2,i)}f_i\middle| Q\in \mathcal{Q} \text{ with }Q\neq J_2\right\}.
    \]
    Here, $\R_{>0}$ is equipped with the LR-coarse structure.
\end{corollary}

Note that, for a set $I$,  $\mathcal{P}\subset \mathcal{P}(I)$ is called a \emph{partition} of $I$ if the following three conditions hold:
\begin{itemize}
    \item $\cup_{P\in \mathcal{P}}P=I$,
    \item $P\neq \emptyset$ for all $P\in \mathcal{P}$.
    \item $P \cap Q =\emptyset$ for all $P,Q\in \mathcal{P}$ with$P\neq Q$.
\end{itemize}

\subsection{An Example of Applications of Theorem \ref{theorem:maintheoremrestrictionl}}

In this section, based on the propositions established in Section \ref{subsection:closeness}, we consider the following examples of a proper action.

\begin{example}
    Let us consider the following closed subgroups of $\rmB (6)$:
    \begin{align*}
        L&:=
        \left\{
            \ell(x,y):=\begin{pmatrix}
                1 & 0 & 0 & 0 & y & x \\
                0 & 1 & x & \frac{x^2}{2} & \frac{x^3}{6} & y \\
                0 & 0 & 1 & x & \frac{x^2}{2} & 0 \\
                0 & 0 & 0 & 1 & x & 0 \\
                0 & 0 & 0 & 0 & 1 & 0 \\
                0 & 0 & 0 & 0 & 0 & 1 \\
            \end{pmatrix}
            \in \rmB(6)
            \, \middle| \, x,y\in \R
        \right\}, \\
        H&:=
        \left\{
            \begin{pmatrix}
                A & 0 \\
                0 & x_{66} \\
            \end{pmatrix}\middle|
            A\in \rmB (5), x_{66}\in \R^\times
        \right\}.
    \end{align*}   
    Let us check that the natural $L$-action on $\rmB(6)/H$ is proper as below.
    The map $\ell:\R^2\to L,\, (a,b)\mapsto \ell(a,b)$ is a topological group isomorphism.
    By assertion \eqref{proposition:grouphom:item:coarseeq} in Proposition \ref{proposition:grouphom}, the map $\ell$ is a coarse equivalence.
    Hence, we may show that the composition $\varphi_{(1,3,1)}\circ\ell:\R^2\to \R$ is coarsely proper.

    By $|\det\ell(x,y)_{(1,2,3),(1,2,3)}|=1$ and $|\det \ell(x,y)_{(1,2,3),(3,5,6)}|=\left|\frac{x^4}{3}\right|+|y^2|$, for each $f\in \R[x]$ with $\deg f\leq 4$ and $g\in \R[y]$ with $\deg g\leq 2$, 
    there exists $C\geq 1$ such that:
    \[
    |f|+|g|\leq C \max\{|\det\ell(x,y)_{(1,2,3),(1,2,3)}|,|\det \ell(x,y)_{(1,2,3),(3,5,6)}|\}.
    \]
    Hence, by Corollary \ref{corollary:maxfull}, we have: 
    \begin{align*}
        \left\|\bigwedge^3 \Phi_{16}(\ell(x,y))\right\|
        &\sim \max\left\{1,\frac{x^4}{3}+y^2, |xy|, \frac{x^2|y|}{2},|x|\left( \frac{x^4}{12}+y^2\right)\right\}\\
        &\sim \max\left\{1,x^4+y^2, |xy|, x^2|y|,|x|\left(x^4+y^2\right)\right\}.
    \end{align*}
    Furthermore, the following two inequalities hold:
    \[
    |xy|\leq \max\{|x|,|y|\}^2\leq \max\{x^4+y^2,1\}, \quad 
    2x^2|y|\leq x^4+y^2.
    \]
    Hence, we can obtain the following by using Corollary \ref{corollary:maxfull} again:
    \begin{align*}
        \left\|\bigwedge^3 \Phi_{16}(\ell(x,y))\right\|
        &\sim \max\left\{1,x^4+y^2, |x|\left( x^4+y^2\right)\right\}\\
        &\sim 1+x^4+|x|^5+(|x|+1)y^2\\
        &\sim 1+x^4+|x|^5+2(|x|+1)y^2.
    \end{align*}

    Next, let us calculate $\max\{\|\bigwedge^3 \Phi_{15}(X)\|,\|\bigwedge^2 \Phi_{15}(X)\||x_{66}|\}$.
    By the equality:
    \[
    \det \ell(x,y)_{(2,3),(4,5)}\cdot |x_{66}|=\frac{x^4}{12},
    \]
    for each $f\in \R[x]$ with $\deg f\leq 4$,
    there exists $D\geq 1$ such that:
    \[
    |f|\leq D\max\{\left|\det \ell(x,y)_{(2,3),(4,5)}\cdot |x_{66}|\right|,1\}.
    \]
    Thus, by Corollary \ref{corollary:maxfull}, we have: 
    \begin{align*}
        \max\left\{\left\|\bigwedge^3 \Phi_{15}(\ell(x,y))\right\|,|x_{66}|\cdot \left\|\bigwedge^2 \Phi_{15}(\ell(x,y))\right\|\right\}
        &\sim \max\left\{1, \frac{x^4}{12},|y|, \frac{x^2|y|}{2}\right\}\\
        &\sim \max\left\{1,x^4,|y|,x^2|y|\right\}.\\
    \end{align*}

    Therefore, the following holds by Proposition \ref{proposition:closecomposition} (see also Example \ref{example:closetoR}):
    \begin{align*}
        \varphi_{(1,3,1)}\circ\ell(x,y)&:=
        \log\left( \frac{\|\bigwedge^3 \Phi_{16}(\ell(x,y))\| }{\max\{\|\bigwedge^3 \Phi_{15}(\ell(x,y))\|,\|\bigwedge^2\Phi_{15}(\ell(x,y))\||x_{66}| \}} \right)\\
        &\sim
        \log \left( \frac{1+x^4+|x|^5+2(|x|+1)y^2}{\max\{1,x^4,|y|,x^2|y|\}}\right).
    \end{align*}

    Put 
    \[
    \psi : \R^2\to \R, \quad (x,y)\mapsto \log \left( \frac{1+x^4+|x|^5+2(|x|+1)y^2}{\max\{1,x^4,|y|,x^2|y|\}}\right).
    \]
    Then, by Proposition \ref{proposition:preserveclose}, $\varphi_{(1,3,1)}\circ \ell$ is coarsely proper if and only if the map $\psi$ is coarsely proper.
    Hence it suffices to show $\liminf_{\|(x,y)\|\to \infty } \psi=\infty$.
    Define the following set:
    \[
    \Omega:=\left\{(x,y)\in \R^2\mid \max\{1,x^4,|y|,x^2|y|\}=\max\{1,x^4,|y|\}\right\}.
    \]
    Then the following immediately holds:
    \[
    \liminf_{\|(x,y)\|\to \infty, (x,y)\in \Omega } \psi =\infty.
    \]
    On the other hand,
    the equality $\R^2\setminus \Omega=\{(x,y)\in \R^2\mid |y|>x^2>1\}$ holds, and we have:
    \begin{align*}
    \liminf_{\substack{\|(x,y)\|\to\infty\\(x,y)\in\mathbb R^2\setminus\Omega}} \log \left( \frac{1+x^4+|x|^5+2(|x|+1)y^2}{x^2|y|}\right)
    &\geq \liminf_{\substack{\|(x,y)\|\to\infty\\(x,y)\in\mathbb R^2\setminus\Omega}} \log \left( \frac{|x|^5+2(|x|+1)y^2}{x^2|y|}\right)\\
    &\geq \liminf_{\substack{\|(x,y)\|\to\infty\\(x,y)\in\mathbb R^2\setminus\Omega}} \log \left( \frac{|x|^3}{|y|}+\frac{2(|x|+1)|y|}{x^2}\right)\\
    &\geq \liminf_{\substack{\|(x,y)\|\to\infty\\(x,y)\in\mathbb R^2\setminus\Omega}} \log \left( 2\sqrt{|x|(|x|+1)}+\frac{(|x|+1)|y|}{x^2}\right)\\
    &\geq \liminf_{\substack{\|(x,y)\|\to\infty\\(x,y)\in\mathbb R^2\setminus\Omega}} \log \left( |x|+|x|+\frac{(|x|+1)|y|}{x^2}\right)\\
    &\geq \liminf_{\substack{\|(x,y)\|\to\infty\\(x,y)\in\mathbb R^2\setminus\Omega}} \log \left( 3(|x||y|+|y|)^\frac{1}{3}\right)=\infty.
    \end{align*}
    Thus, $\liminf_{\|(x,y)\|\to \infty} \psi=\infty$, and it implies that $\varphi_{(1,3,1)}\circ\ell:\R^2\to \R$ is coarsely proper.
    Therefore, the natural $L$-action on $\rmB(6)/H\cong \R^5$ is proper.
\end{example}

\begin{remark}
    In the setting of the above example, let us consider another closed subgroup of $\rmB(6)$ as below:
    \[
    L':=
        \left\{
            \begin{pmatrix}
                1 & 0 & 0 & 0 & -y & x \\
                0 & 1 & x & \frac{x^2}{2} & \frac{x^3}{6} & y \\
                0 & 0 & 1 & x & \frac{x^2}{2} & 0 \\
                0 & 0 & 0 & 1 & x & 0 \\
                0 & 0 & 0 & 0 & 1 & 0 \\
                0 & 0 & 0 & 0 & 0 & 1 \\
            \end{pmatrix}
            \in \rmB(6)\middle| x,y\in\R 
        \right\}.
    \]
    The only difference between $L$ and $L'$ is the sign of the $(1,5)$-entry.
    However, the $L'$-action on $\rmB(6)/H$ is not proper.
    Yoshino \cite{Yoshino2005counterLips} showed this and thereby provided a counter example to Lipsman's conjecture.
\end{remark}

\section{Remarks on closed subgroups of $\rmB(n)$}\label{section:subgroups}

The following theorem is obtained from Propositions \ref{proposition:composition}, \ref{proposition:grouphom} and Theorem \ref{theorem:maintheoremimage}:

\begin{theorem}
    Fix $n\in \Z_{\geq 1}$ and a set $\Lambda$ with $\Lambda_0(n)\subset \Lambda\subset \Lambda_1(n)$ (see Section \ref{subsection:coarsemapproper}). 
    Let $G$ be a closed subgroup of $\rmB(n)$.
    Then the composition $\mu_\Lambda\circ \iota_G : G\to \R^{ \Lambda}$ is coarse.
    Here, $G$ is equipped with the LR-coarse structure, $\R^{\Lambda}$ is equipped with the bounded coarse structure (see Example \ref{example:coarse}) and $\iota_G$ is the inclusion map from $G$ to $\rmB(n)$.
    In particular, for closed subgroups $L,H$ of $G$, 
    if the pair $(\mu_\Lambda\circ \iota_G(L), \mu_\Lambda\circ \iota_G(H))$ is asymptotically disjoint in $(\R^{ \Lambda},\mathcal{E}_d)$, 
    the natural $L$-action on $G/H$ is proper (see Corollary \ref{corollary:coarsemapproper}). 
\end{theorem}

\begin{remark}
    Although the above theorem is a sufficient condition for the properness of the natural $L$-action on $G/H$, it cannot be said to be a necessary condition in general.
    In fact, Example \ref{example:nilpotentproper} will be an example that cannot be determined by the above theorem:
\end{remark}

Finally, let us focus on the case $G$ is the group of all upper triangle matrix such that each diagonal entry is $1$.
For $n\in \Z_{\geq 2}$, put 
$\rmN(n):=\{X\in \rmB(n)\mid x_{ii}=1 \text{ for all }i\in [n]\}$, and define the map $\alpha_0$ as below:
\[
    \alpha_0:\rmN(n)\to \R,\quad X\mapsto \log(\|X\|),
\]
and, for each $i\in [n-1]$, put
\[
    \alpha_i:\rmN(n)\to \R, \quad X\mapsto x_{i,i+1}.
\]
Note that $\alpha_0=\mu_{(1,n,1,1)}\circ \iota_{\rmN(n)}$.
Here, $\iota_{\rmN(n)}$ is the inclusion map from $\rmN(n)$ to $\rmB(n)$.
the following theorem holds:

\begin{theorem}\label{theorem:nilpotent}
    Fix $n\in \Z_{\geq 2}$. Then the following three conditions hold:
    \begin{enumerate}
        \item\label{theorem:nilpotent:item:0} The map $\alpha_0=\mu_{(1,n,1,1)}\circ \iota_{\rmN(n)}$ is coarse.
        \item\label{theorem:nilpotent:item:i} The map $\alpha_i$ is controlled for all $i\in [n-1]$.
        \item\label{theorem:nilpotent:item:alpha} Fix a subset $I\subset [n-1]\cup\{0\}$ with $0\in I$ and $\Lambda\subset \Lambda_1(n)$.
        Then the following map is coarse:
        \[
        (\alpha_I, \mu_{\Lambda}):\rmN(n)\to \R^{I\sqcup \Lambda}, \quad X\mapsto (\alpha_i(X))_{i\in I}\oplus (\mu_\Lambda(X)).
        \]
        Here, in case $\Lambda=\emptyset$, put
        $(\alpha_I, \mu_{\Lambda}):=\alpha_I$.
    \end{enumerate}
\end{theorem}

\begin{proof}
    Let us check assertion \eqref{theorem:nilpotent:item:0}.
    the controlledness of $\alpha_0$ directly follows from Lemma \ref{lemma:logsubmultiplicative} (see also Example \ref{example:submultiplicative}).
    Further, for any $R\geq 1$, the following holds:
    \[
    \alpha_0^{-1}(E_R[0])\subset \left\{
    \begin{pmatrix}
        1& x_{12} &\cdots & x_{1n}\\
        0& 1 &\ddots & \vdots \\
        \vdots& \ddots &\ddots & x_{n-1n} \\
        0& \cdots &0 & 1 \\
    \end{pmatrix}
    \middle|
    |x_{ij}|\leq e^R \text{ for all } 1\leq i<j\leq n
    \right\}.
    \]
    Hence, the map $\alpha_0$ is coarsely proper and thus, $\alpha_0$ is coarse.
    By Proposition \ref{proposition:grouphom}, we can obtain assertion \eqref{theorem:nilpotent:item:i}.
    Finally, assertion \eqref{theorem:nilpotent:item:alpha} holds by Proposition \ref{proposition:products}.
\end{proof}

Let us apply the above theorem to the following test case:

\begin{example}\label{example:nilpotentproper}
    Let us consider the following closed subgroups of $\rmN(3)$:
    \[
    L:=\left\{ \ell(t):=
    \begin{pmatrix}
        1 & t & \frac{t^2}{2} \\
        0 & 1 & t \\
        0 & 0 & 1 \\
    \end{pmatrix}
    \middle|
    t\in \R
  \right\},\quad
   H:=\left\{ h(s):=
    \begin{pmatrix}
        1 & \frac{s}{4} & \frac{s^2}{4} \\
        0 & 1 & 2s \\
        0 & 0 & 1 \\
    \end{pmatrix}
    \middle|
    s\in \R
  \right\}
    \]
    Then the $L$-action on $\rmN(3)/H$ is proper.
    In fact,
    put $\widetilde{\alpha}:=(\alpha_0,\alpha_1,\alpha_2)$
    and  consider the images of $\widetilde{\alpha}$ as below:
    \begin{align*} 
    \widetilde{\alpha}(L)=\left\{\left(\alpha_0(\ell(t)),t,t\right)\middle| t\in \R\right\},\,
    \widetilde{\alpha}(H)=\left\{\left( \alpha_0(h(s)),\frac{s}{4},2s\right)\middle| s\in \R\right\}.
    \end{align*}
    it suffices to show that the pair $(\widetilde{\alpha}(L),\widetilde{\alpha}(H))$
    is asymptotically disjoint.
    Fix $R\geq 1$.
    For any element $\widetilde{\alpha}\circ h(s)\in E_R[\widetilde{\alpha}(L)]\cap \widetilde{\alpha}(H)$, we have the following:
    \[
   |s|\leq \left|2s-\frac{s}{4}\right|\leq |t-2s|+\left|t-\frac{s}{4}\right|\leq R+R=2R.
    \]
    Hence, the inclusion $E_R[\widetilde{\alpha}(L)]\cap \widetilde{\alpha}(H)\subset \widetilde{\alpha}\left(\{h(s)\in H\mid s\in [-2R,2R] \}\right)$ holds, and the intersection $E_R[\widetilde{\alpha}(L)]\cap \widetilde{\alpha}(H)$ is bounded.
    Thus, the pair $(\widetilde{\alpha}(L),\widetilde{\alpha}(H))$ is asymptotically disjoint.
    Therefore, the $L$-action on $\rmN(3)/H$ is proper.

    Next, let us consider $\iota_{\rmN(3)}(L)$ and $\iota_{\rmN(3)}(H)$ in $\rmB(3)$.
    For the sake of brevity, we write $\iota_{\rmN(3)}(L)$ as $L$ and $\iota_{\rmN(3)}(H)$ as $H$.
    Then $L$-action on $\rmB(3)/H$ is not proper.
    In fact, 
    \[
    L=\begin{pmatrix}
        2 & 0 & 0\\
        0 & \frac{1}{2} & 0\\
        0 & 0 & 1\\
    \end{pmatrix}
    \cdot H \cdot 
    \begin{pmatrix}
        2 & 0 & 0\\
        0 & \frac{1}{2} & 0\\
        0 & 0 & 1\\
    \end{pmatrix}^{-1}.
    \]
\end{example}

\section*{ACKNOWLEDGMENTS}
The author would like to give heartfelt thanks to Takayuki Okuda for his encouragement
to write this paper. We are also indebted to Temma Aoyama, Naotsugu Chinen, Tomohiro Fukaya, Tatsuro Hikawa, Kazuki Kannaka, Hikozo Kobayashi, Toshiyuki Kobayashi, Akira Kubo, Takumi Matsuka, Muneto Miyaji, Shunsuke Miyauchi, Yosuke Morita, Akifumi Nakada, Hideto Nakashima, Kento Ogawa, Shinichi Oguni, Atsumu Sasaki, Hiroshi Tamaru, Koichi Tojo, Kotaro Mine, Takamitsu Yamauchi, Ibuki Yonezawa for many helpful comments. 
I am especially grateful to Professors Hideyuki Ishi and Ali Baklouti for organizing the 8th Tunisian-Japanese Conference and for providing such a wonderful opportunity for stimulating mathematical exchange. I would also like to thank all the organizers and participants for inspiring atmosphere that made the conference a truly memorable occasion.
This work is supported by JST SPRING, Grant Number JPMJSP2132.

\providecommand{\bysame}{\leavevmode\hbox to3em{\hrulefill}\thinspace}
\providecommand{\MR}{\relax\ifhmode\unskip\space\fi MR }
\providecommand{\MRhref}[2]{%
  \href{http://www.ams.org/mathscinet-getitem?mr=#1}{#2}
}
\providecommand{\href}[2]{#2}

\end{document}